\documentclass[twoside,12pt]{article}
\usepackage[letterpaper]{geometry}
\usepackage{verbatim}
\usepackage{amsmath}
\usepackage{bbm}
\usepackage{ltexpprt}
\usepackage{hyperref}
\usepackage{setspace}

\usepackage[T1]{fontenc}
\usepackage{newtxtext,newtxmath}
\usepackage{tikz}
\usepackage{authblk}
\usepackage{graphicx}
\usepackage[font=normalsize]{caption}
\usepackage[small]{subfigure}

\usepackage[dvipsnames]{xcolor}
\usepackage{tkz-euclide,tikzscale}
\usetikzlibrary{arrows,circuits.ee.IEC}
\tikzset{ac source/.style={
  circuit symbol lines,
  circuit symbol size = width 2 height 2,
  shape = generic circle IEC,
  /pgf/generic circle IEC/before background={
    \pgfpathmoveto{\pgfpoint{-0.8pt}{0pt}}
    \pgfpathsine{\pgfpoint{0.4pt}{0.4pt}}
    \pgfpathcosine{\pgfpoint{0.4pt}{-0.4pt}}
    \pgfpathsine{\pgfpoint{0.4pt}{-0.4pt}}
    \pgfpathcosine{\pgfpoint{0.4pt}{0.4pt}}
    \pgfusepath{stroke}
  },
  transform shape
}}
\newcommand{\solidorangeline}{\raisebox{2pt}{\tikz{\draw[-,Orange,line width = 2pt](0,0) -- (4mm,0);}}}
\newcommand{\solidgreenline}{\raisebox{2pt}{\tikz{\draw[-,Green,line width = 2pt](0,0) -- (4mm,0);}}}
\newcommand{\solidredline}{\raisebox{2pt}{\tikz{\draw[-,Maroon,line width = 2pt](0,0) -- (4mm,0);}}}
\newcommand{\solidblueline}{\raisebox{2pt}{\tikz{\draw[-,blue,line width = 2pt](0,0) -- (4mm,0);}}}
\newcommand{\solidpurpleline}{\raisebox{2pt}{\tikz{\draw[-,Periwinkle,line width = 2pt](0,0) -- (4mm,0);}}}
\newcommand{\dottedredline}{\raisebox{2pt}{\tikz{\draw[red,line width = 2pt,dash pattern=on 2pt off 2pt](0,0) -- (4mm,0);}}}
\newcommand{\dottedblackline}{\raisebox{2pt}{\tikz{\draw[black,line width = 2pt,dash pattern=on 2pt off 2pt](0,0) -- (4mm,0);}}}
\newcommand{\bm}[1]{{\mathbbm{#1}}}

\newcommand{\Bcal}{{\cal B}}

\newcommand{\Rset}{\mathbb{R}}

\newcommand{\Iset}{\mathbb{I}}
\newcommand{\Jset}{\mathbb{J}}
\newcommand{\Uset}{\mathbb{U}}
\newcommand{\vnorm}[1]{\left\|#1\right\|}
\newcommand{\argmin}{\mathop{\rm argmin}}
\newcommand{\op}{\text{op}}

\begin{document}

\newcommand\relatedversion{}

\title{\Large Spatially-Coupled Network RNA Velocities:\\
A Control-Theoretic Perspective\relatedversion}

\author[1,2]{Boya Hou}
\author[1,2]{Maxim Raginsky}
\author[3]{Abhishek Pandey} 
\author[1,2]{Olgica Milenkovic}

\affil[1]{\small Department of Electrical and Computer Engineering, University of Illinois Urbana-Champaign}
\affil[2]{\small Carl R. Woese Institute for Genomic Biology, University of Illinois Urbana-Champaign}
\affil[3]{\small AbbVie Pharmaceuticals}

\date{}

\maketitle

\begin{abstract} \small\baselineskip=9pt \noindent 
RNA velocity is an important model that combines cellular spliced and unspliced RNA counts to infer dynamical properties of various regulatory functions. Despite its wide applicability and many variants used in practice, the model has not been adequately designed to directly account for both intracellular gene regulatory network interactions and spatial intercellular communications. Here, we propose a new RNA velocity approach that jointly and directly captures two new network structures: an intracellular gene regulatory network (GRN) and an intercellular interaction network that captures interactions between (neighboring) cells, with relevance to spatial transcriptomics. We theoretically analyze this two-level network system through the lens of control and consensus theory. In particular, we investigate network equilibria, stability, cellular network consensus, and optimal control approaches for targeted drug intervention.
\end{abstract}



\section{Introduction and Problem Formulation}
RNA velocity is a modeling concept used to infer cellular differentiation trajectories from bulk and single-cell RNA sequencing data~\cite{la2018rna}. The key idea behind the model is to couple the counts of unspliced and spliced mRNA molecules into a dynamical system, and define velocity as an indicator of the future state of spliced mRNA given its unspliced molecular counts. Specifically, given a single cell and a single gene, the evolution of unspliced RNA $u(t)$ and spliced RNA $s(t)$ is captured by two ordinary differential equations (ODEs) of the form 
\begin{align}
\begin{aligned}
\frac{\mathrm{d} u(t)}{\mathrm{d} t} = \alpha(t) - \beta u(t), \quad
\frac{\mathrm{d} s(t)}{\mathrm{d} t} = \beta u(t) 
- \gamma s(t),
\end{aligned}
\label{eq.0}
\end{align}
where $\alpha(t)$ stands for the time-dependent transcription rate (i.e., the ``expression'' rate at which DNA is read to produce mRNA), $\beta$ represents the splicing rate (i.e.,  the rate at which mRNA is modified via alternative splicing), and $\gamma$ equals the degradation rate (i.e., the rate at which mature RNA is used up for translation into proteins). RNA velocity itself is defined as $v(t)= \frac{d s(t)}{\mathrm{d} t}$~\cite{la2018rna} so that a positive velocity implies that the expression of the underlying gene is increasing, while a negative RNA velocity indicates an opposite trend. In addition, $v(t)=0$ implies that the replication/splicing system is in an equilibrium\footnote{Although many cellular mechanisms inherently exist and operate in nonequilibrium states, there are equally many examples of systems that operate in equilibria, including bacteriophage lambda lysogenic maintenance circuits, drosophila segment polarity network, etc.~\cite{fang2020nonequilibrium}.}. 

Two of the most frequently used RNA velocity models are Velocyto~\cite{la2018rna} and its extended version, termed ScVelo~\cite{bergen2020generalizing}. Velocyto relies on the assumption that genes in a cell have reached a steady-state expression level. At the equilibrium, the ratio of the unspliced RNA to spliced RNA of a gene is determined by the ratio of the degradation and splicing rates. Velocyto quantifies RNA velocity as the deviation from the steady-state ratio. ScVelo~\cite{bergen2020generalizing}, on the other hand, extends the process of estimating RNA velocity to transient systems by using a dynamical model. Although it relaxes steady-state assumptions, the model only describes the transcription dynamics of genes in a single cell. The recently proposed GraphVelo model~\cite{chen2025graphvelo} refines the RNA velocity estimates by projecting them onto the tangent space of a low-dimensional manifold of the single-cell data, and extends RNA velocity estimates to multi-modal single-cell data. Perhaps the most related model to ours is TFVelo~\cite{li2024tfvelo}. It extends the gene expression model to incorporate the influence of transcription factors. More precisely, TFVelo uses a sine function to describe the regulatory behavior, but fails short of explicitly modeling the influence of regulatory genes on the transcription rates. In summary, models of the form described in Equations~\eqref{eq.0} only capture the transcriptional dynamics of a single gene within a single cell, and abstract various network controls through the rate parameters, which are usually inferred from data. This indirect inference/modeling approach may hence be compromised by limited and noisy data evidence.

This paper generalizes all the above lines of work by \emph{explicitly including information about gene regulatory networks} and extending the model to address \emph{cellular populations within spatial networks that work towards a functional consensus}. 

To enable modeling the influence of interventions, one needs to explicitly account for the regulatory relationships between genes that control the transcription process. Furthermore, to describe communications between cells, it is desirable to introduce consensus constraints that, in practice, can be explained via spatial transcriptomics data. Towards this end, we assume that each gene expression is controlled by a gene regulatory network (GRN) comprising $n_g$ regulatory genes\footnote{For simplicity, the proposed model is mostly tailored towards transcription factor networks.}. Since gene expression can be either positively (activation) or negatively (repression) regulated~\cite{khammash2022cybergenetics}, we use two nonnegative matrices $W^{+}$ and $W^{-}$ to represent the weighted directed regulatory networks in which $W^{+}$ captures positive and $W^{-}$ negative regulations. 
We also enforce $W^{+}_{gq} \cdot W^{-}_{gq}=0$ for each pair $g,q$. In simple terms, $W_{gq}^{+} \neq 0$ captures the fact that gene $q$ positively influences the expression of gene $g$, so that it is automatically implied that $W_{gq}^{-}=0$ since gene $q$ cannot repress gene $g$ in this case. A similar explanation holds for $W_{gq}^{-} \neq 0$. Note that these weights can be estimated using not only expression data but any multiomics source of ``interaction'' evidence. The sample complexity required for the identification process is discussed in~\cite{sontag2003differential}.

Under the above assumptions, the evolution of unspliced RNA $u^g(t)$ and spliced RNA $s^g(t$) can be described as
\begin{align}
\begin{aligned}
\frac{\mathrm{d} u^g}{\mathrm{d} t} = \alpha^g \ \frac{\kappa+ \sum_{q=1}^{n_g} W_{gq}^{+} s^q(t)}{\kappa+ \sum_{q=1}^{n_g} W_{gq}^{-} s^q(t)}- \beta^g u^g(t),\quad
\frac{\mathrm{d} s^g}{\mathrm{d} t} = \beta^g u^g(t) - \gamma^g s^g(t), 
\end{aligned}
\label{eq.GRN}
\end{align}
where $\kappa \geq 0$ is a constant.  Here, the rate parameters represent the \emph{basal expression, splicing, and degradation rates} of individual genes, but how much of that basal rate is utilized is controlled by the network of transcription factors indexed by $g$. This allows for direct accounting of the influence of individual transcription factors, as well as intervention efficiency. For $g = 1,\cdots, n_g$, we use a nonlinear (rational) gene control function of the form
\begin{align}
R_g(s):= \frac{\kappa+ \sum_{q=1}^{n_g} W_{gq}^{+} s^q}{\kappa+ \sum_{q=1}^{n_g} W_{gq}^{-} s^q}, \qquad s:=[s^1,\dots,s^q,\dots,s^{n_g}]^\top
\label{eq.R}
\end{align}
which aggregates all positive and negative regulatory effects within the numerator and denominator, respectively. Our modeling choice is governed by two considerations: the connection of the model to Hill functions~\cite{alon2019introduction}, which are rational functions, albeit with more general polynomial terms, and analytical tractability. For completeness, we provide a review of the Hill function model in Section~\ref{app:Hill} of the Supplementary Information (SI).

The model in Equation~\ref{eq.GRN} only considers a single cell, while communication and synchronization of activities across cells, as encountered in population dynamics models~\cite{brauer2012mathematical}, are overlooked. Cells typically communicate via diffusion of signaling molecules, such as hormones, lipids, and ions, etc., as well as proteins, which may be viewed as spliced RNA products. Importantly, communication-enabling extracellular vesicles contain RNA that can be transcribed in target cells, and other evidence suggests that RNAs can act as hormones~\cite{o2020rna,bayraktar2017cell,kehr2018long,wu2002signaling}.  Hence, to mitigate the issue of incorporating cellular consensus, one can instead revise the model by considering a network of $n_c$ cells, each controlled by $n_g$ genes. The  \emph{spatially-coupled} RNA network velocity model we propose to study takes the form:
\begin{align}
\begin{aligned}
\frac{\mathrm{d} u_i^g}{\mathrm{d} t} &= \alpha_i^g \ \frac{\kappa+ \sum_{q=1}^{n_g} W_{gq}^{+} s_i^q(t)}{\kappa+ \sum_{q=1}^{n_g} W_{gq}^{-} s_i^q(t)}- \beta_i^g u_i^g(t),\\
\frac{\mathrm{d} s_i^g}{\mathrm{d} t} &= \beta_i^g u_i^g(t) 
-\gamma_i^g s^g_i(t) 
+ {c} \sum_{j=1}^{n_c} A_{ij}\left( s^g_j(t) - s^g_i(t) \right).
\end{aligned}
\label{eq.consensus}
\end{align}
Here, the superscript $g$ indexes genes, while subscripts such as $i$ and $j$ index cells. As before, $\kappa$ is a constant, and so is $c$ as well. The term ${c} \sum_{j=1}^{n_c} A_{ij}( s^g_j(t) - s^g_i(t) )$ models the consensus network, in which $A_{ij}$ describe the intrercell (communication) adjacency matrix, with $A_{ij}\neq 0$ if cells $i$ and $j$ are exchanging signaling molecules. Note that in order to make the model tractable for analysis, we used spliced RNA concentrations as proxies for the corresponding protein concentrations, with the scaling factor $c$ succinctly capturing the molecular ``conversion'' loss.

The goal of our work is to analyze the GRN and the joint GRN-consensus RNA velocity models from the perspective of control theory. In particular, we examine the conditions under which the dynamical systems allow for an equilibrium, and when the equilibria are stable. Furthermore, we investigate intervention (perturbation) strategies for GRNs with the purpose of examining the potential effect of gene knockouts or drugs on the behavior of the coupled dynamical models. We view the problem of designing such intervention strategies under various constraints as \emph{minimum-time optimal control problems}. To the best of our knowledge, this control-theoretic formulation has not been proposed before.

The paper is organized as follows. In Section~\ref{sec:GRN}, we derive conditions for the existence of an equilibrium and its stability for the single-cell GRN-driven RNA velocity model. In Section~\ref{sec:GRNConsensus}, we extend this line of analysis by accounting for the consensus term in the model. Direct and indirect intervention models are analyzed through the lens of optimal control theory in Section~\ref{sec:intervention}, resulting in explicit results for GRN-driven and numerical findings for spatially-coupled consensus models. Technical background on nonnegative dynamical systems can be found in Section~\ref{app:nonnegative} of the SI. The proofs of all lemmas and theorems are also given in the SI.

\section{Equilibria and Stability of GRN-Driven RNA Velocity Models} 
\label{sec:GRN}

We start by showing that the ODE model of network RNA velocity in Equation~\eqref{eq.GRN} is consistent with the GRN structure. For brevity, denote the right-hand sides of Equation~\eqref{eq.GRN} by $f_g^u$ and $f_g^s$, so that $\frac{d u^g}{\mathrm{d} t} = f_g^u$ and $\frac{d s^g}{\mathrm{d} t} = f_g^s$.  For each gene $g$, we let 
$$N_g:= \kappa+ \sum_{q=1}^{n_g} W_{gq}^{+} s^q, \quad D_g:=\kappa+ \sum_{q=1}^{n_g} W_{gq}^{-} s^q.$$ 
When either $W_{gq}^{+}$ or $W_{gq}^{-}$ is nonzero, gene $q$ directly regulates gene $g$, and the regulatory effect is encoded in
\begin{align}
\frac{\partial f_g^u }{ \partial s^q}
= \frac{\alpha^j \left(W^{+}_{ji} D_j -  W^{-}_{ji} N_j\right)}{\left(D_j\right)^2}.
\end{align}
When $W_{gq}^{+}>0$, we have $\frac{\partial f_g^u }{ \partial s^q}>0$, implying gene $q$ is an activator of gene $g$; and when $W_{gq}^{-}>0$, we have $\frac{\partial f_g^u }{ \partial s^q}<0$, indicating gene $q$ is a repressor of gene $g$. Hence, our ODE~\eqref{eq.GRN} is consistent with the GRN.

While consistency with the GRN is sufficient for the analysis in this paper, indirect influences between genes that arise through multi-step pathways can, in principle, be analyzed using the \textit{constant sign property} (CSP) framework introduced in~\cite{kang2020graph}. For a given pair of genes $(q,g)$, the idea behind the CSP is to examine whether the influence of $q$ on $g$ through the ODE dynamics is well-defined and monotonic. In this case, one first identifies the shortest path(s) in the GRN that connect $u^q$ to $u^g$~\footnote{According to~\cite[Definition 4]{kang2020graph}, one first defines a molecular graph whose vertices are internal molecular classes, and then merges the molecular states which belong to the same gene to recover the GRN via~\cite[Proposition 1]{kang2020graph}. Since our ODE model only involves $u^g$ and $s^g$ and since, for each gene $g$, $u^g$ directly affects $s^g$, we adapt the analysis in~\cite{kang2020graph} to directly examine the level of influence of genes}. For each such path $\pi$, one computes the product of the first-order partial derivatives of the underlying molecular functions along that path. Notice that in our ODE model~\eqref{eq.GRN}, $u^q$ directly affects $s^q$ as captured by $\frac{d s^g}{\mathrm{d} t} = \beta^g u^g(t) - \gamma^g s^g(t)$. Hence, the product over the shortest path $\pi$ takes the form $\prod_{(i,j) \in \pi} \frac{\partial f_i^s }{ \partial u^i} \times \frac{\partial f_j^u }{ \partial s^i}$. When there are multiple shortest paths, one proceeds as follows. Denote the collection of shortest paths by $P(q,g)$. Then, take the sum of all products along the shortest paths to obtain the sum-product defined in ~\cite[Definition 7]{kang2020graph}. More specifically, define the sum-product quantity $Q$ as
\begin{align}
\begin{aligned}
    Q(q,g,u,s) 
    &= \sum_{\pi \in P(q,g)} \prod_{(i,j) \in \pi} \frac{\partial f_i^s }{ \partial u^i} \times \frac{\partial f_j^u }{ \partial s^i} \\
    &= \sum_{\pi \in P(q,g)} \prod_{(i,j) \in \pi} \left( \beta^i 
    \times \frac{\alpha^j \left(W^{+}_{ji} D_j -  W^{-}_{ji} N_j\right)}{\left(D_j\right)^2} \right).
\end{aligned}
\label{eq.CSP}
\end{align}
The sum-product monotonicity is defined as the sign of $Q$ as $B(q,g,u,s) = \text{sign} \left(Q(q,g,u,s)\right)$. If the shortest path is unique, or all shortest paths have the same sign\footnote{Also, note that due to the stochastic nature of transcription, one often only uses information about the ``sign'' of interaction: activating or repressing.}, $B(q,g,u,s)$ is constant over the entire state space, indicating the indirect influence on gene $q$ on gene $g$ is monotonic, and hence the system is \emph{globally CSP}, per \cite[Definition 7]{kang2020graph}. In this case, the indirect influence can be represented by a single directed edge, and as the rational function $R_g(s)$ in~\eqref{eq.R} is real and analytic, and therefore smooth, one can invoke~\cite[Proposition 1]{kang2020graph} to conclude that the underlying ODE model is consistent with a single signed (directed) graph. 

Next, given that the network dynamics~\eqref{eq.GRN} is nonlinear, and regulatory effects cannot be directly interpreted based on the sign of the weighted adjacency matrix (e.g., weights). The constant sign property~\eqref{eq.CSP} provides one way to explain the notion of positive and negative feedback, since when there is only one shortest path between gene $q$ and gene $g$, say $q \to g$, $B(q,g,u,s)$ reflects the regulatory effect encoded in $W^{\pm}_{gq}$ as
\begin{align}
\begin{aligned}
    B(q,g,u,s) 
    =& \text{sign} \left( \frac{\partial f_q^s }{ \partial u^q} \times \frac{\partial f_g^u }{ \partial s^q} \right)\\
    =& \text{sign} \left( \beta^q \times \frac{\alpha^g \left( W^{+}_{gq} D_g - W^{-}_{gq} N_g\right)}{\left(D_g\right)^2} \right)\\
    =&\begin{cases}
    +1, & \text{if } W^{+}_{gq}>0, \\
    -1, & \text{if } W^{-}_{gq}>0.
    \end{cases}
\end{aligned}
\end{align}

In what follows, we present a direct characterization of regulatory effects through the incremental gain of the nonlinear function $R_g(s)$ in Equation~\eqref{eq.gain}, which also reveals its dependence on $W^{+}$ and $W^{-}$ \cite{sepulchre2019feedback}. Consider two spliced RNA configurations $s,\hat{s}$ that agree in all coordinates except for $q$, and let
$N'_g=\kappa+ \sum_{q'=1}^{n_g} W_{gq'}^{+} \hat{s}^{q'}$, $ D'_g=\kappa+ \sum_{q'=1}^{n_g} W_{gq'}^{-} \hat{s}^{q'}$. Write $\delta s^q:= s^q - \hat{s}^q$. Then, the incremental gain of $R_g$ due to a change in $s_q$ equals
\begin{align}
\begin{aligned}
\frac{R_g(s) - R_g(\hat{s})}{s^q - \hat{s}^q}
= \frac{N_g D'_g - D_g N'_g}{D_g D'_g \ \delta s^q}  
= \frac{N_g \left(D_g - W_{gq}^{-} \delta s^q\right) - D_g \left(N_g - W_{gq}^{+} \delta s^q \right)}{D_g D'_g \ \delta s^q} 
= \frac{D_g W_{gq}^{+} - N_g W_{gq}^{-}}{D_g D'_g}.
\end{aligned}
\label{eq.gain}
\end{align}
Hence, the incremental gain is positive if $W^{+}_{gq} > 0$ (in which case $W^-_{gq} = 0$), and negative if $W^{-}_{gq} > 0$ (in which case $W^+_{gq} = 0$). 
As a result, the incremental gain directly describes how changes in $s^q$ affect $R_g(s)$, and subsequently, influence the dynamics of $u^g$. Positive (resp., negative) incremental gain indicates the presence of positive (resp., negative) feedback \cite{sepulchre2019feedback}.
 
With this in mind, we turn our attention to an analysis of equilibria and stability of the single-cell network RNA velocity model.

\subsection{Equilibria and Stability.}

Let $u:=[u^1,\cdots,u^{n_g}]^\top$, $s:=[s^1,\cdots,s^{n_g}]^\top$, $\alpha = \text{diag}\left(\alpha^1,\cdots,\alpha^{n_g} \right)$, $\beta := \text{diag}\left(\beta^1,\cdots,\beta^{n_g} \right)$, and $\gamma := \text{diag}\left(\gamma^1,\cdots,\gamma^{n_g} \right)$, . Furthermore, let $R(s):=[R_1(s),\cdots, R_{n_g}(s)]^\top$.
Equation~\eqref{eq.GRN} can then be compactly rewritten as
\begin{align}
\begin{aligned}
\frac{\mathrm{d} u}{\mathrm{d} t} = \alpha R(s) - \beta u, \quad
\frac{\mathrm{d} s}{\mathrm{d} t} = \beta u - \gamma s.
\end{aligned}
\label{eq.GRN.vec}
\end{align}
Recall that for a nonlinear system of the form $\frac{\mathrm{d}x}{\mathrm{d} t} = f(x)$, a point $x_e$ is an equilibrium of the system if $f(x_e) = 0$. Our main analytical results for the networked dynamics~\eqref{eq.GRN.vec} are listed below.

The first theorem provides a sufficient condition for the existence of an equilibrium point with all coordinates nonnegative. Recall that for a matrix $X \in \Rset^{n \times n}$ with eigenvalues $\lambda_1,\cdots,\lambda_n$, the spectral radius $\rho(X)$ of $X$ is $\rho(X) = \max_{1\leq i \leq n} \left| \lambda_i \right|$. We write $X \succ 0$ to indicate that $X$ is positive definite.
\begin{theorem}
Suppose that $\beta \succ 0$ and $\gamma \succ 0$, and define $\Lambda := \frac{1}{\kappa} \gamma^{-1} \alpha W^{+}$, where $W^+ := [W^+_{gq}]^{n_q}_{g,q=1}$ and $W^- := [W^-_{gq}]^{n_g}_{g,q=1}$.
The networked dynamics admits an equilibrium point $(u^*,s^*) \in \Rset^{n_g}_+ \times \Rset^{n_g}_+$ if the spectral radius of $\Lambda$ satisfies $\rho\left(\Lambda\right)<1$.
\label{thm.ss}
\end{theorem}
For $\kappa=1$, one can see from Theorem~\ref{thm.ss} that a sufficient condition for the existence of equilibria is that 
$\rho\left(\gamma^{-1} \alpha W^{+}\right)<1$. Roughly speaking, this indicates that the positive regulation in the GRN cannot overwhelm the degradation. 

We next study the stability of the networked dynamics. We first consider the special case when $W^{-}$ is a zero matrix, i.e., no gene acts as a repressor of any other gene. In this case, we have a linear system 
\begin{align}
\begin{aligned}
\frac{\mathrm{d} u^g}{\mathrm{d} t} = \alpha^g  \left(\kappa+ \sum_{q=1}^{n_g} W_{gq}^{+} s^q\right) - \beta^g u^g, \quad 
\frac{\mathrm{d} s^g}{\mathrm{d} t} = \beta^g u^g - \gamma^g s^g,
\end{aligned}
\label{eq.GRN.linear}
\end{align}
which may be unstable depending on the system parameters. 

\begin{lemma} 
Suppose that $\alpha \succ 0$, $\beta \succ 0$, and the matrix $\gamma - \alpha W^+$ is positive definite. Suppose further that there are no repressors, i.e., $W^{-}_{gq} = 0$, $\forall g,q = 1,\cdots, n_g$. Then the networked dynamics is stable if, for all $g$, 
$\gamma^g > \beta^g > \alpha^g \sum_{h} W^{+}_{gh}$.
\label{lemma.stability0}
\end{lemma}
The above lemma indicates that, when there are no repressors, the activator‑only model becomes a positive-feedback-regulated system. Compared with Theorem \ref{thm.ss}, an extra condition on the splicing rate $\beta$ is needed to ensure the stability of equilibria. 

When $W^{-}$ is not the all-zero matrix, the system is nonlinear and may or may not be stable depending on how strong the negative feedback is (i.e., how large the incremental gains are). Our next result provides a sufficient condition for system stability via the Lyapunov direct method (see Section~\ref{app:nonnegative} of SI for key notions of stability for nonnegative dynamical systems).
 
\begin{theorem}
Suppose that the conditions of Theorem~\ref{thm.ss} hold, and consider a positive definite function
\begin{align*}
V(u,s):= \frac{1}{2} \vnorm{u - u^*}^2_2 + \frac{1}{2} \vnorm{s-s^*}^2_2.
\end{align*}
Let $\vnorm{s}_1= \sum_{q=1}^{n_g} |s^q|$, and
suppose that there exists a $\delta >0$ such that, for each $s \ge 0$, $\min_g [W^{-}s]_g \geq \delta \vnorm{s}_1$. Furthermore, write
\begin{align*}
c_1:= \max_{g,q} \left(W_{gq}^{+}, W_{gq}^{-}\right), 
\quad 
\omega:=n_g \max\left(\frac{c_1}{\kappa}, \frac{c_1^3}{4 \delta \kappa (c_1 - \delta)}\right).
\end{align*}
If for all $g = 1,\cdots, n_g$, 
$\beta^g > \frac{\omega \vnorm{\alpha}}{2}$ and 
$\gamma^g > \frac{\omega \vnorm{\alpha}}{2} 
+ \frac{{\beta^g}^2}{4(\beta^g - \frac{\omega \vnorm{\alpha}}{2} )}$, 
then, $\dot{V}(u,s) < 0$ for all $(u,s) \neq (u^*,s^*)$. That is, $V(u,s)$ is a valid Lyapunov function, and $(u^*,s^*) \in \Rset^{n_g}_+ \times \Rset^{n_g}_+$ is a unique equilibrium of \eqref{eq.GRN.vec} which is globally asymptotically stable. 
\label{theorem.stablity}
\end{theorem}
In the above theorem, the condition $\min_g [W^- s]_g \ge \delta \|s\|_1$ guarantees that each gene is repressed according to the total spliced RNA level in the network. Since the repressors correspond to negative feedback, this adds robustness through negative feedback regulation. The set of sufficient conditions describes the interactions between the splicing rate $\beta$, transcription rate $\alpha$, degradation rate $\gamma$, and the constant $\omega$ that depends on the network.

Finally, we remark that in the special case of a single gene,  the dynamics is described by the set of linear ODEs as
\begin{align}
\frac{\mathrm{d} }{\mathrm{d} t} \begin{bmatrix} u\\ s \end{bmatrix}  
= \begin{bmatrix} -\beta & 0 \\
\beta & -\gamma \end{bmatrix}
\begin{bmatrix} u\\ s\end{bmatrix} 
+\begin{bmatrix}  \alpha \\0\end{bmatrix} .
\end{align}
The eigenvalues of $\begin{bmatrix} -\beta & 0 \\
\beta & -\gamma \end{bmatrix}$ are $\lambda_1=-\beta$ and $\lambda_2=-\gamma$. Hence, as long as $\beta>0$ and $\gamma>0$, we have $\lambda_1<0$ and $\lambda_2<0$. This implies the unique equilibrium $(u^*,s^*) = \left(\frac{\alpha}{\beta},\frac{\alpha}{\gamma} \right)$ is always stable.

\section{Spatially-Coupled GRN-Driven RNA Velocity Models} \label{sec:GRNConsensus}

We now turn our attention to the GRN-based velocity model coupled with a spatial network model, per Equation~\ref{eq.multinetwork.dynamics}, repeated below for convenience. Recall that $n_c$ denotes the number of cells, each of which has an internal regulatory network of $n_g$ genes, and that for all genes $g \in \{1,\dots,n_g\}$ and cells $i \in \{1,\dots,n_c\}$,
\begin{align}
\begin{aligned}
\frac{\mathrm{d} u_i^g}{\mathrm{d} t} &= \alpha_i^g \ \frac{\kappa+ \sum_{q=1}^{n_g} W_{gq}^{+} s_i^q(t)}{\kappa+ \sum_{q=1}^{n_g} W_{gq}^{-} s_i^q(t)}- \beta_i^g u_i^g(t),\\
\frac{\mathrm{d} s_i^g}{\mathrm{d} t} &= \beta_i^g u_i^g(t) 
-\gamma_i^g s^g_i(t) 
+ {c} \sum_{j=1}^{n_c} A_{ij}\left( s^g_j(t) - s^g_i(t) \right).
\end{aligned}
\label{eq.multinetwork.dynamics}
\end{align}
The unspliced and spliced RNA concentrations, $u_i^g$ and $s_i^g$, must remain nonnegative for all time $t$. This is obvious for the single-cell model~\eqref{eq.GRN}. We now show that this property still holds for the spatially-coupled RNA network velocity model \eqref{eq.multinetwork.dynamics} when we include the consensus term ${c} \sum_{j=1}^{n_c} A_{ij}\left( s^g_j(t) - s^g_i(t)\right)$ that captures the intercellular coupling. The next lemma guarantees that the model~\eqref{eq.multinetwork.dynamics} is essentially nonnegative, thus biologically meaning as $u_i^g$ and $s_i^g$ for all genes $g \in \{1,\dots,n_g\}$ and cells $i \in \{1,\dots,n_c\}$ can never take negative values.

\begin{lemma} \label{lemma.essentially.nonnegative}
Consider the spatially-coupled GRN-driven RNA velocity model~\eqref{eq.multinetwork.dynamics}. 
Assume that all the parameters $\alpha_i^g, \beta_i^g, \gamma_i^g, \kappa, c, W_{gq}^{\pm}, A_{ij}$ are nonnegative for each cell  $i = 1, \ldots, n_c$ and gene $g = 1, \ldots, n_g$.
Then \eqref{eq.multinetwork.dynamics} is essentially nonnegative, i.e.,  if the initial condition satisfies $u_i^g(0) \geq 0$, $s_i^g(0) \geq 0$ for all $i, g $, then the solution satisfies $u_i^g(t) \geq 0$, $s_i^g(t) \geq 0$ for all $t \geq 0$. 
\end{lemma}

\subsection{Equilibria and Stability Analysis.} 
Similar to what was presented in the previous section, the following results characterize the existence of equilibria and the stability of the spatially-coupled GRN-driven RNA velocity model.

For each cell $i \in \{1,\dots,n_c\}$, define $s_i = [s_i^1,\cdots s_i^{n_g}]^\top \in \Rset^{n_g}$,
$\alpha_i = \text{diag}\left(\alpha_i^{1},\cdots, \alpha_i^{n_g} \right) \in \Rset^{n_g \times n_g}$,
$\beta_i = \text{diag}\left(\beta_i^{1},\cdots, \beta_i^{n_g} \right)\in \Rset^{n_g \times n_g}$,
$\gamma_i = \text{diag}\left(\gamma_i^{1},\cdots, \gamma_i^{n_g} \right)\in \Rset^{n_g \times n_g}$, 
and $s= [s_1^\top,\cdots s_{n_g}^\top]^\top \in \Rset^{n_g \cdot n_c}$. 

\begin{theorem}
Suppose that $\beta_i \succ 0$ and $\gamma_i \succ 0$ for all cells  $i = 1, \ldots, n_c$. Define the matrix $\Lambda \in \Rset^{n_g \cdot n_c \times n_g \cdot n_c}$ as
\begin{align}
\Lambda := {\rm diag}\left( \frac{1}{\kappa} \gamma_1^{-1} \alpha_1 W^{+}, \cdots, \frac{1}{\kappa} \gamma_{n_c}^{-1} \alpha_{n_c} W^{+}\right)    
+ {c} \left(A \otimes I_{n_g} \right) {\rm diag}\left(\gamma_1^{-1},\cdots,\gamma_{n_c}^{-1}  \right).
\end{align}
The spatially coupled dynamics \eqref{eq.multinetwork.dynamics} admits an equilibrium $(u^*,s^*)$ in the nonnegative orthant if the spectral radius of $\Lambda$ satisfies $\rho(\Lambda) <1$. 
\label{thm.ss-multi}
\end{theorem}

Just like in the preceding section, we can give sufficient conditions for stability in the absence of repressors:

\begin{lemma}
Suppose the condition of Theorem \ref{thm.ss-multi} holds. When there are no repressors, i.e., $W^{-}_{gq} = 0$, $\forall g,q = 1,\cdots, n_g$, the equilibrium of the network dynamics is stable if for all cells $i \in \{1,\dots,n_c\}$ and genes $g \in \{1,\dots,n_g\}$, 
$\gamma_i^g  > \beta_i^g > \frac{\alpha_i^g}{\kappa} \sum_{h} W^{+}_{gh}$.
\label{lemma.stability0-multi}
\end{lemma}

We next consider the case where gene expression is also allowed to be negatively regulated, i.e., $W^{-}$ is not the zero matrix. 
\begin{theorem}
Suppose that the conditions of Theorem~\ref{thm.ss-multi} hold. Consider the positive definite function \begin{align}
    V(u,s):= \frac{1}{2} \sum_{i=1}^{n_c} \sum_{g= 1}^{n_g} \left( \left( u_i^g - {u_i^g}^* \right)^2 + \left( s_i^g - {s_i^g}^*\right)^2 \right).
\end{align}
Suppose that, for each cell $i$, there exists a real number $\delta_i >0$ such that $\min_g [W^{-}s]_g \geq \delta_i \vnorm{s_i}_1$ for all $s_i$ with nonnegative coordinates. Let $c_1:= \max_{g,q} \left(W_{gq}^{+}, W_{gq}^{-}\right)$, and $\omega_i:=\sqrt{n_g} \max\left(\frac{c_1}{\kappa}, \frac{c_1^3}{4 \delta_i \kappa (c_1 - \delta_i)}\right)$.
Let $\gamma_i:= \max_{g} \gamma_i^g$.
If for all cells $i \in \{1,\dots,n_c\}$ and genes $g \in \{1,\cdots, n_g\}$, 
$\beta_i^g > \frac{\omega_i \vnorm{\alpha_i}_F }{2}$, and 
$\gamma_i^g >  \frac{\omega_i \vnorm{\alpha_i}_F }{2}
+ \frac{({\beta_i^g})^2}{4(\beta_i^g - \frac{\omega_i \vnorm{\alpha_i}_F }{2})},$ 
then, $\dot{V}(u,s) < 0$ for all $(u,s) \neq (u^*,s^*)$. That is, $V(u,s)$ is a valid Lyapunov function, and $(u^*,s^*)$ is a unique globally asymptotically stable equilibrium.
\label{theorem.stablity-multi}
\end{theorem}

\subsection{Cellular-Network Consensus.}

We now turn to the analysis of cellular-network consensus due to spatial coupling in the spliced dynamics. Recall that the spliced dynamics in the spatially-coupled RNA velocity model are given by
\begin{align}
\begin{aligned}
\frac{\mathrm{d} s_i^g}{\mathrm{d} t} = \beta_i^g u_i^g(t) 
-\gamma_i^g s^g_i(t) + {c} \sum_{j=1}^{n_c} A_{ij}\left( s^g_j(t) - s^g_i(t) \right), 
\end{aligned}
\label{eq.multinetwork.s}
\end{align}
for all cells $i \in \{1,\dots,n_c\}$ and genes $g \in \{1,\dots,n_g\}$.
Let $s^g := [s_1^g,\cdots, s_{n_c}^g]^\top\in \Rset^{n_c}$, $u^g := [u_1^g,\cdots, u_{n_c}^g]^\top\in \Rset^{n_c}$, $B^g := \text{diag}(\beta_1^g,\cdots, \beta_{n_c}^g)$, $\Gamma^g := \text{diag}(\gamma_1^g,\cdots,\gamma_{n_c}^g)$, $D := {\rm diag} (A\bm{1})$, and let $L := D - A$ denote the (unnormalized) graph Laplacian of the cellular network \cite{mesbashi2010graphs}. Then, we have
\begin{align}
\begin{aligned}
\frac{\mathrm{d} s^g}{\mathrm{d} t} = B^g u^g(t) 
-\Gamma^g s^g(t) - {c} L s^g(t). 
\end{aligned}
\end{align}

Define the linear projection operator $P_c := \frac{1}{n_c} \bm{1} \bm{1}^\top$. We can decompose $s^g(t)$ into two parts: the mean field (average) $\overline{s}^g(t) := P_c s^g(t)\in \Rset^{n_c}$, and the deviation $\tilde{s}^g(t):=s^g(t) - \overline{s}^g(t) \in \Rset^{n_c}$. For the mean-field dynamics, we have
\begin{align}
\begin{aligned}
\frac{\mathrm{d} \overline{s}^g}{\mathrm{d} t}
=& P_c
\left(B^g u^g(t) -\Gamma^g s^g(t) \right) 
-P_c L s^g(t). 
\end{aligned}
\end{align}
Since $\bm{1}^\top L = 0$, the mean-field dynamics reduces to
\begin{align}
\begin{aligned}
\frac{\mathrm{d} \overline{s}^g}{\mathrm{d} t} 
= P_c \left( B^g u^g(t) -\Gamma^g s^g(t) \right).
\end{aligned}
\end{align}
In order to establish if convergence to a cellular consensus is possible, we establish the following theorem, which characterizes the evolution of the deviation $\tilde{s}^g$.

\begin{theorem}
Assume that $u^g(t)$,$s^g(t)$ are bounded for all genes $g=1,\cdots,n_g$:
\begin{align}
\sup_{t \ge 0} \vnorm{u^g(t)}_2 < \infty, \quad \sup_{t \ge 0} \vnorm{s^g(t)}_2 < \infty.
\end{align}
Then, asymptotically, the deviation can be upper-bounded according to
\begin{align}
\limsup_{t \to \infty} \vnorm{\tilde{s}^g(t)}_2^2 \leq \frac{1}{c \lambda_2(L)}Z_m^g,
\label{eq.fluctation}
\end{align}    
where $\lambda_2(L)$ is the second smallest eigenvalue of the graph Laplacian and $0<Z_m^g<\infty$ is a constant equal to $Z_m^g:=\max_{t \ge 0} \vnorm{B^g u^g(t) -\Gamma^g s^g(t)}_2$.
\label{thm.consensus}
\end{theorem}

Theorem~\ref{thm.consensus} only assumes uniform boundedness of $u^g$ and $s^g$. When the stability conditions of Theorem~\ref{theorem.stablity-multi} are met, we automatically have uniform boundedness as well:

\begin{corollary}
Under conditions in Theorem \ref{theorem.stablity-multi}, the fluctuation component is bounded as
\begin{align}
\limsup_{t \to \infty} \vnorm{\tilde{s}^g}_2^2 \leq \frac{1}{c \lambda_2(L)}Z_m^g.
\end{align}
\end{corollary}

Theorem~\ref{thm.consensus} indicates that the deviation becomes smaller when the cellular network is strongly connected, which corresponds to the case when $\lambda_2(L)$ is large. Traditional models of cell-to-cell communication networks encode local spatial contacts or signaling relationships~\cite{armingol2020deciphering, cang2023screening} using lattice-like graphs which may exhibit poor global connectivity, long diffusion times, and may be susceptible to localized perturbations. From a systems biology point of view, it is reasonable to consider expander-like graph connectivity patterns, as they can capture key qualitative features of biological signaling networks such as fast propagation, robustness to cell loss, and resilience to communication bottlenecks~\cite{hoory2006expander}. Furthermore, expander graphs are sparse yet highly connected structures that exhibit large spectral gaps and rapid information mixing~\cite{hoory2006expander}, so that signals originating from a small subset of cells can quickly and robustly influence the global cellular population, even in the presence of noise and stochastic failures. For $d$-regular expander graphs, the Alon--Boppana bound and the existence of Ramanujan graphs~\cite{alon1986eigenvalues, nils1994survey} establish that
\begin{align*}
\lambda_2(L) \ge d - 2\sqrt{d-1}.
\end{align*}
Using $d=12$, which would model cells as perfect spheres and enforce the optimal kissing-number constraint~\cite{conway1999sphere} for $3$-dimensional spaces, we have $\lambda_2(L) \ge 12 - 2\sqrt{11} \approx 5.36$, reflecting a strong form of algebraic connectivity that remains constant regardless of the size of the cell population.

\section{Intervention as Minimum-Time Optimal Control Problem} 
\label{sec:intervention}
\subsection{Controlled GRN-Driven RNA Velocity.}
\label{sec.controlled.GRN}
 
We next study how a targeted (drug) intervention alters the dynamics of the system. In particular, we investigate direct and indirect target interventions. To this end, we formulate the drug intervention problem as a time-optimal control problem.

We consider the setup where we can control individual genes by directly modulating their expression levels and thus affecting their ability to influence other genes. We further assume that only the positive-feedback genes are subject to control.  Let $z^q(t)$ be the control applied to gene $q$ that takes values in a given interval $\Uset := [\underline{\zeta},\overline{\zeta}]$ where $0\leq \underline{\zeta} \leq \overline{\zeta}$. We then have the following controlled-network RNA velocity dynamics described by equations
\begin{align}
\begin{aligned}
\frac{d u^g}{\mathrm{d} t} = \alpha^g \ 
\frac{\kappa
+\sum_{p \neq q}^{n_g}  W_{gp}^{+} s^p(t)
+ z^q(t) W_{gq}^{+} s^q(t)}
{\kappa+ \sum_{p=1}^{n_g} W_{gp}^{-} s^p(t)}- \beta^g u^g(t),\quad
\frac{d s^g}{\mathrm{d} t} = \beta^g u^g(t) - \gamma^g s^g(t), 
\end{aligned}
\label{eq.GRN.controlled}
\end{align}
where $W^{+}_{gg} > 0$.
Let $D_g(s) = \kappa+ [W^{-} s]_g$ and define
\begin{align}
R_g^\circ(z^q,s):= \frac{\kappa+ \sum_{p \neq q}^{n_g}  W_{gp}^{+} s^p(t)
+ z^q(t) W_{gq}^{+} s^q(t)}{D_g(s)},
\quad g = 1,\cdots, n_g.
\label{eq.R.controlled}
\end{align}
Also, let $R^\circ(z^q,s):=[R_1^\circ(z^q,s),\cdots, R_{n_g}^\circ(z^q,s)]^\top$, and use  $\Iset$ to denote the set of target genes among $1,\cdots, n_g$. The goal is to design the controller $z^q(t)$ such that $s^r$ for $r \in \Iset$ are driven to the desired value $s_{\text{target}}^r$ for all $r \in \Iset$ as fast as possible. Hence, we need to solve the following \textit{minimum-time optimal control problem}:
\begin{align}
\begin{aligned}
\min_{z^q}& \int_{0}^{T} 1 \ {\mathrm{d} t} \\
\text{subject to } \ 
& \frac{\mathrm{d} u}{\mathrm{d} t} = \alpha R^\circ(z^q,s) - \beta u, \quad 
\frac{\mathrm{d} s}{\mathrm{d} t} = \beta u - \gamma s,\\
& u(0) = u_0, \quad s(0) = s_0, \\
& s^r(T) = s_{\text{target}}^r, \quad r \in \Iset, \\
& z^q(t) \in \Uset, \quad \forall t \in [0,T].
\end{aligned}
\label{eq.GRN.time-optimal-control}
\end{align}

This problem is a special case of the fixed-endpoint control problem. It can be addressed using the Pontryagin maximum principle (PMP)~\cite{vinter2010optimal} that characterizes the optimal controller $z^q_\star$. 
Let $\lambda_u^g, \lambda_s^g$ be the costates and $\lambda_u, \lambda_s$ denote the concatenated costate vector, for which the Hamiltonian $H \left( u, s, \lambda_u, \lambda_s, z^q \right)$ equals
\begin{align}
H \left(u, s, \lambda_u, \lambda_s, z^q \right) = 1 + \sum_{g}^{n_g} \lambda_u^g\left(\alpha^g R_g^\circ(z^q,s) - \beta^g u^g \right) + \sum_{g}^{n_g} \lambda_s^g \left(\beta^g u^g - \gamma^g s^g \right).
\label{eq.GRN.H}
\end{align}
The costate dynamics are specified by
\begin{align}
\begin{aligned}
\frac{\mathrm{d} \lambda_u^g}{\mathrm{d} t} =& - \frac{\partial H}{\partial u^g}
=  \beta^g \lambda_u^g - \beta^g \lambda_s^g, \\
\frac{\mathrm{d} \lambda_s^g}{\mathrm{d} t} =& - \frac{\partial H}{ \partial s^g}
= - \lambda_u^g \alpha^g \frac{\partial R_g^\circ(z^q,s)}{\partial s^g} 
- \sum_{r \neq g} {\lambda_u}^r \alpha^r \frac{\partial R_r^\circ(z^q,s)}{\partial s^g} 
+ \lambda_u^g \gamma^g \lambda_s^g,
\end{aligned}
\label{eq.GRN.costate}
\end{align}
for all $g = 1,\cdots, n_g$.
In the above equation, in order to compute $\frac{\partial R^\circ_r}{\partial s^g}$, denote the numerator and denominator in $R^\circ_r$ by
$N^r:=\kappa+ \sum_{p \neq q}^{n_g}  W_{rp}^{+} s^p(t)
+ z^q(t) W_{rq}^{+} s^q(t) $, and $D^r:= \kappa+ \sum_{p=1}^{n_g} W_{rp}^{-} s^p,$ respectively.
We have
\begin{align}
\begin{aligned}
\frac{\partial R^\circ_r}{\partial s^g}
=\begin{cases}
\frac{z^q W_{rq}^{+} D^r -  N^r W_{rq}^{-} }{(D^r)^2}, 
& \text{if } g = q, \\
\frac{W_{rg}^{+} D^r -  N^r W_{rg}^{-} }{(D^r)^2},
& \text{if } g \neq q.
\end{cases}
\end{aligned}
\end{align}

Since \eqref{eq.GRN.time-optimal-control} imposes the constraint that $s^r$ reaches a specific state at the final time $T$, the boundary conditions for the costates are
\begin{align}
\begin{aligned}
&\lambda_u^g(T)=0, \quad g = 1,\cdots, n_g \\
&\lambda_s^g(T)=0, \quad g \notin \Iset, \\
&\lambda_s^r(T) \ \text{free},  \quad r \in \Iset.
\end{aligned}    
\label{eq.GRN.boundary}
\end{align}

In addition, because the terminal time $T$ is free, the transversality condition from the optimal control theory implies that, for all $t$,
\begin{align}
H \left(u_\star(t), s_\star(t), \lambda_{u,\star}(t), \lambda_{s,\star}(t), z^q_\star(t) \right) = 0.
\end{align}

Observe that in the Hamiltonian~\eqref{eq.GRN.H}, the term that depends on $z^q$ equals 
\begin{align}
\begin{aligned}
\sum_{g}^{n_g} \lambda_u^g \alpha^g \frac{\kappa+ z^q (t) W_{gq}^{+} s^q(t)}{D_g(s)} 
=& \sum_{g}^{n_g} \lambda_u^g \alpha^g \frac{\kappa}{D_g(s)} 
+ z^q(t) s^q(t)  \sum_{g}^{n_g} \lambda_u^g \alpha^g \frac{W_{gq}^{+}}{D_g(s)}.
\end{aligned}
\end{align}
Since $s^q (t)$ is nonnegative, let  $\Psi(t,{\lambda_u}) = \sum_{g}^{n_g} \lambda_u^g \alpha^g  \frac{W_{gq}^{+}}{D_g(s)}$ which is a weighted sum of the costates $\lambda_u^g$ for $g = 1,\cdots,n_g$. Then, a necessary condition for the optimal controller $z^q_\star(t)$ is that it is a bang-bang controller of the form:
\begin{align}
\begin{aligned}
z^q_\star(t)
=& \argmin_{z^q} H\left(u, s, \lambda_u, \lambda_s, z^q \right) \\
=& \argmin_{z^q} \ z^q (t) \Psi(t,{\lambda_u}) \\
=& \begin{cases} 
\overline{\zeta} & \text{if} \ \Psi(t,{\lambda_u}) <0, \quad \text{and} \quad s^q(t) \neq 0, \\ 
\underline{\zeta} & \text{if} \ \Psi(t,{\lambda_u}) >0, \quad \text{and} \quad s^q(t) \neq 0,\\ 
\text{Undecided} & \text{if}  \ \Psi(t,{\lambda_u}) = 0, \quad \text{or} \quad s^q(t) = 0.
\end{cases}
\end{aligned}
\label{eq.GRN.switch}
\end{align}

Since PMP provides only a necessary condition for optimality, it does not by itself determine whether the target state $s_f$ is reachable from an initial state $s_0$. To address reachability, we compute the Lie bracket (which indicates what types of dynamic changes are possible)~\cite{vinter2010optimal} for the controlled system~\eqref{eq.GRN.controlled}, and show it can be connected to the previously discussed molecular distance defined in~\cite{kang2020graph}. 

For clarity, we focus on the case where the target is a single gene $r$. Let $x = [u^\top, s^\top]^\top \in \Rset^{2n_g}$, and note that the relevant nonlinear dynamics can be written in a control-affine form,
\begin{align}
\frac{d x}{\mathrm{d} t} = f(x) + G(x) z^q(t),
\end{align}
where 
$f(x) = \begin{bmatrix} f_u(x) \\ f_s(x)\end{bmatrix}$ 
is the drift vector field with 
\begin{align*}
f_u^g (x)= \alpha^g \ 
\frac{\kappa
+\sum_{p \neq q}^{n_g}  W_{gp}^{+} s^p(t)}
{D_g(s)}- \beta^g u^g(t),
\end{align*}
and 
\begin{align*}
f_s^g(x) =  \beta^g u^g(t) - \gamma^g s^g(t),
\end{align*}
for each $g = 1,\cdots, n_g$. The control vector field $G(x) = \begin{bmatrix} G_u(x) \\ G_s(x)\end{bmatrix}$ has a nonzero entry for each index $g$ where $W_{gq}^{+} \neq 0,$  and the actual values equal $[G_u(x)]_g = \frac{\alpha^g W_{gq}^{+} s^q(t)}{D_g(s)}.$ Furthermore, $G_s \equiv 0$.

To examine how an intervention on gene $q$ propagates from $u$ to $s$ and over the GRN as a whole, we leverage the molecular graph introduced in~\cite[Definition 4]{kang2020graph} whose nodes are either $u^g$ or $s^g$ for all $g = 1,\cdots n_g$. 
Computing successive Lie brackets reveals how the control input propagates through the GRN. Since $G_s \equiv 0$, the first-order bracket is given by 
\begin{align}
\begin{aligned}
\left[f,G\right]= D G(x) f(x) - Df(x) G(x) 
= \begin{bmatrix} \frac{\partial G_u}{\partial s} \cdot f_s(x) \\ 0 \end{bmatrix}
+ \begin{bmatrix} - \frac{\partial f_u}{\partial u} \cdot G_u(x) \\ - \frac{\partial f_s}{\partial u} \cdot G_u(x) \end{bmatrix},
\end{aligned}
\end{align}
where $D$ denotes the Jacobian. Here, for the $u$-component in $\left[f,G\right]$ at index $h_0 = 1,\cdots,n_g$, we have
\begin{align}
\begin{aligned}
\left[f,G\right]_{u^{h_0}}
=& \sum_{k}^{n_g} \frac{\partial {[G_u]}_{h_0}}{\partial s^k}  f_s^k(x)  - \sum_{k}^{n_g} \frac{\partial f_u^{h_0}}{\partial u^k} \left[G_u(x)\right]_k \\
=& \sum_{k}^{n_g} \frac{\partial {\left[G_u\right]}_{h_0}}{\partial s^k}  f_s^k(x)  +  \sum_{k}^{n_g} \beta^{h_0} \delta_{{h_0} k} \left[G_u(x)\right]_k \\
=& \sum_{k}^{n_g} \frac{\partial {\left[G_u\right]}_{h_0}}{\partial s^k}  f_s^k(x)  +  \beta^h  \left[G_u(x)\right]_{h_0}. 
\end{aligned}
\end{align}
The $s$-component of $\left[f,G\right]$ at index $h_0 = 1,\cdots,n_g$  is given by
\begin{align}
\begin{aligned}
\left[f,G\right]_{{s}^{h_0}}= - \frac{\partial f_s}{\partial u} \cdot G_u(x)  
= - \sum_{k}^{n_g} \frac{\partial f_s^{h_0}}{\partial u^k}  \left[G_u(x)\right]_k 
= - \beta^{h_0} \left[G_u(x)\right]_{h_0}.
\end{aligned}
\end{align}
Hence, the first Lie bracket $\left[f,G\right]$ introduces nonzero $s$-entry at each gene $h_0$ with $W_{{h_0} q}^{+} \neq 0$, capturing the immediate effect of the control $z^q$ applied to $u^q$ on downstream spliced states $s^{h_0}$.

Next, let $v_1:=\left[f,G\right]$, and compute the second-order Lie bracket $v_2:=\left[f, v_1 \right]= D v_1(x) f(x) - Df(x) v_1(x)$.
Since we are interested in how an intervention $z^q(t)$ at gene $q$ propagates, and we have already computed the one-step propagation from $u^{h_0}$ to $s^{h_0}$ with $h_0$ such that $W_{{h_0}q}^{+} \neq 0$, we only need to focus on the effect of $s^{h_0}$ on $u^{h_1}$ within the two-step neighborhood of $u^q$. In other words, we focus on the $u$-component of $\left[f, v_1 \right]$. For ${h_1} = 1,\cdots, n_g$, we have
\begin{align}
\begin{aligned}
\left[v_2\right]_{{u}^{h_1}}
=& \sum_{k}^{n_g} \frac{\partial {\left[v_{1}\right]}_{u^{h_1}}}{\partial u^k} f_s^k  
+ \sum_{k}^{n_g} \frac{\partial {\left[v_{1}\right]}_{u^{h_1}}}{\partial s^k} f_s^k  
- \sum_{k}^{n_g} \frac{\partial f_u^{h_1}}{\partial u^k}   {\left[v_{1}\right]_{u^k}}
- \sum_{k}^{n_g} \frac{\partial f_u^{h_1}}{\partial s^k}  {\left[v_{1}\right]_{s^k}} \\
=& \underbrace{ \sum_{k}^{n_g} \frac{\partial {\left[v_{1}\right]}_{u^{h_1}}}{\partial u^k} f_s^k 
+ \sum_{k}^{n_g} \frac{\partial {\left[v_{1}\right]}_{u^{h_1}}}{\partial s^k} f_s^k
+ \beta^{h_1}  
{\left[v_{1}\right]}_{u^{k}}}_{=:{Q}_{2,u}^{h_1}}
+ \sum_{k}^{n_g} B_{{h_1}k} \left( \beta^k \left[G_u(x)\right]_k \right) ,
\end{aligned}
\end{align}
where $B_{hk}:= \frac{\partial f_u^h}{\partial s^k}$ encodes the effect of $s^k$ on $f_u^h$.
The above equation implies that for any ${h_1}$ such that $B_{{h_1}k} \neq 0$ and $W_{kq}^{+} \neq 0$ and for which $u^{h_1}$ is two molecular-graph-steps away from $u^q$, the effect of $z^q$ on $u^{h_1}$ manifests itself via the last term $\sum_{k}^{n_g} B_{{h_1}k} \left( \beta^k \left[G_u(x)\right]_k \right)$. 
The $s$-component of $\left[f, v_1 \right]$ can be computed in a similar way. Since, by the structure of the dynamics, along a directed path in the molecular graph $u^q \to s^q \to u^{h_0} \to s^{h_0} \to u^{h_1} \to \cdots$, the coupling between $u$ and $s$ alternates at each bracket order.
To avoid redundancy, we therefore skip the explicit computation of the $s$-components for even orders and of the $u$-components for odd orders.

The $s$-component of the third-order Lie bracket $v_3:=\left[f, v_2 \right]= D v_2(x) f(x) - Df(x) v_2(x)$ is computed as
\begin{align}
\begin{aligned}
\left[v_3\right]_{s}^{h_2}
=& \sum_{k}^{n_g} \frac{\partial {\left[v_{2,s}\right]}^{h_2}}{\partial u^k} f_u^k  
+ \sum_{k}^{n_g} \frac{\partial {\left[v_{2,s}\right]}^{h_2}}{\partial s^k} f_s^k  
- \sum_{k}^{n_g} \frac{\partial f_s^{h_2}}{\partial u^k}   {\left[v_{2,u}\right]}^{k}
- \sum_{k}^{n_g} \frac{\partial f_s^{h_2}}{\partial s^k}  {\left[v_{2,s}\right]}^{k},
\end{aligned}
\end{align}
where the effect of $G_u$ enters through the third term as
\begin{align}
\begin{aligned}
\sum_{k}^{n_g} \frac{\partial f_s^{h_2}}{\partial u^k}   {\left[v_{2,u}\right]}^{k}
&= \sum_{k}^{n_g} \beta^{h_2} \delta_{h_2 k} {\left[v_{2,u}\right]}^{k} \\
&= \sum_{k}^{n_g} \beta^{h_2} \delta_{h_2 k} 
\left( {Q}_{2,u}^{k} + \sum_{j}^{n_g} B_{k j}  \beta^j \left[G_u(x)\right]_j \right) \\
&= \beta^{h_2}  {Q}_{2,u}^{k} 
+ \beta^{h_2} \sum_{j}^{n_g} B_{k j}  \beta^j \left[G_u(x)\right]_j.
\end{aligned}
\end{align}
Thus, we have 
$\left[v_3\right]_{s}^{h_2}
= {Q}_{3,s}^{h_2}
- \beta^{h_2} \sum_{j}^{n_g} B_{k j}  \beta^j \left[G_u(x)\right]_j$,
where ${Q}_{3,s}^{h_2}:=\sum_{k}^{n_g} \frac{\partial {\left[v_{2,s}\right]}^{h_2}}{\partial u^k} f_u^k  
+ \sum_{k}^{n_g} \frac{\partial {\left[v_{2,s}\right]}^{h_2}}{\partial s^k} f_s^k  
- \sum_{k}^{n_g} \frac{\partial f_s^{h_2}}{\partial u^k}   {\left[v_{2,u}\right]}^{k}
-  \beta^{h_2}  {Q}_{2,u}^{k}$.

We can repeat the calculations for the fourth-order Lie bracket, $v_4:=\left[f, v_3 \right]= D v_3(x) f(x) - Df(x) v_3(x)$, and determine its $u$-component according to
\begin{align}
\begin{aligned}
\left[v_4\right]_{u}^{h_3}
&= \sum_{k}^{n_g} \frac{\partial {\left[v_{3,u}\right]}^{h_3}}{\partial u^k} f_s^k  
+ \sum_{k}^{n_g} \frac{\partial {\left[v_{3,u}\right]}^{h_3}}{\partial s^k} f_s^k  
- \sum_{k}^{n_g} \frac{\partial f_u^{h_3}}{\partial u^k}   {\left[v_{3,u}\right]}^{k}
- \sum_{k}^{n_g} \frac{\partial f_u^{h_3}}{\partial s^k}  {\left[v_{3,s}\right]}^{k} \\
&= \sum_{k}^{n_g} \frac{\partial {\left[v_{3,u}\right]}^{h_3}}{\partial u^k} f_s^k 
+ \sum_{k}^{n_g} \frac{\partial {\left[v_{3,u}\right]}^{h_3}}{\partial s^k} f_s^k
+ \beta^{h_3}  
{\left[v_{3,u}\right]}^{k} \\
&\quad - \sum_{k}^{n_g} B_{{h_3}k} \left({Q}_{3,s}^{k}
- \beta^{h_2} \sum_{j}^{n_g} B_{k j}  \beta^j \left[G_u(x)\right]_j\right) \\
&= \underbrace{ \sum_{k}^{n_g} \frac{\partial {\left[v_{3,u}\right]}^{h_3}}{\partial u^k} f_s^k 
+ \sum_{k}^{n_g} \frac{\partial {\left[v_{3,u}\right]}^{h_3}}{\partial s^k} f_s^k
+ \beta^{h_3}  
{\left[v_{3,u}\right]}^{k}
- \sum_{k}^{n_g} B_{{h_3}k} {Q}_{3,s}^{k}
}_{:={Q}_{4,u}^{h_3}} \\
& \quad +  \sum_{k}^{n_g} B_{{h_3}k}  \beta^{k} \left(\sum_{j}^{n_g} B_{k j}  \beta^j \left[G_u(x)\right]_j\right).
\end{aligned}
\end{align}
Here, the last term $\sum_{k}^{n_g} B_{{h_3}k}  \beta^{k} \left(\sum_{j}^{n_g} B_{k j}  \beta^j \left[G_u(x)\right]_j\right)$ captures the effect of the control on nodes that are four molecular distance away.

This iterative structure shows that the earliest Lie bracket in which the control appears corresponds to the molecular distance between gene $q$ and $r$. Consequently, the existence of a directed path in the molecular graph is necessary for the reachability of $s^r_{\text{target}}$ in the controlled GRN.


\subsection{Controlled Spatially-Coupled GRN-Driven RNA Velocity}
We conclude our analysis by considering drug intervention for spatially coupled intercellular GRN networks. The controlled dynamics are
\begin{align}
\begin{aligned}
\frac{\mathrm{d} u_i^g}{\mathrm{d} t} &= \alpha_i^g \ \frac{\kappa
+ \sum_{p \neq q}^{n_g} W_{gp}^{+} s_i^p(t)
+\left[\delta_i z^q(t)+  \left( 1 - \delta_i \right)  \right] W_{gq}^{+} s_i^q(t)}{\kappa+ \sum_{p=1}^{n_g} W_{gp}^{-} s_i^p(t)}- \beta_i^g u_i^g(t),\\
\frac{\mathrm{d} s_i^g}{\mathrm{d} t} &= \beta_i^g u_i^g(t) 
-\gamma_i^g s^g_i(t) 
+ c \sum_{j=1}^{n_c} A_{ij}\left( s^g_j(t) - s^g_i(t) \right),
\end{aligned}
\label{eq.multinetwork.controlled}
\end{align}
for all cells $i \in \{1,\dots,n_c\}$ and genes $g \in \{1,\dots,n_g\}$, where each $z^q(t)$ denotes the control input that targets gene $q$. The binary random variables $\delta_i \in\{0,1\}$ for $i = 1,\cdots,n_c$ specify whether cell $i$ is affected by the drug: when $\delta_i=1$, cell $i$ is affected by $z^q(t)$, and when $\delta_i=0$, cell $i$ is not affected by the drug and evloves according to its nominal dynamics. In other words, the drug does not necessarily influence every cell in the population, and $\delta_i$ encodes such spatial heterogeneity. 

Define 
\begin{align}
\overline{R}_i^g: = \frac{\kappa+ \sum_{p \neq q}^{n_g} W_{gp}^{+} s_i^p(t) + \left[\delta_i z^q(t)+  \left( 1 - \delta_i \right)  \right]  W_{gq}^{+} s_i^q(t)}{\kappa+ \sum_{p=1}^{n_g} W_{gp}^{-} s_i^p}.
\end{align}

Similar to the previous section, we can formulate the minimum-time optimal control problem to solve for $z^q_\star (t)$ as follows:
\begin{align}
\begin{aligned}
\min_{z^q}& \int_{0}^{T} 1 \ {\mathrm{d} t} \\
\text{subject to } \ & \frac{\mathrm{d} {u}_i^g}{\mathrm{d} t} = \alpha_i^g \overline{R}_i^g - \beta_i^g u_i^g(t), \\
&\frac{\mathrm{d} s_i^g}{\mathrm{d} t} = \beta_i^g u_i^g(t) 
-\gamma_i^g s^g_i(t) 
+ c \sum_{j=1}^{n_c} A_{ij}\left( s^g_j(t) - s^g_i(t) \right),\\
& u_i^g(0) = u_{i,0}^g, \quad s_i^g(0) = s_{i,0}^g, \quad, \forall i = 1,\cdots,n_c, \quad \forall g = 1,\cdots, n_g, \\
& s_j^r(T) = s_{j,\text{target}}^r, \quad r \in \Iset, \quad j \in \Jset, \\
& z^q(t) \in \Uset, \quad \forall t \in [0,T],
\end{aligned}
\label{eq.intra.time-optimal-control}
\end{align}
where $s_{i,\text{target}}^r$ is the targeted final value for gene $r \in \Iset$ and cell $i \in \Jset$, and $\Uset = [\underline{\zeta},\overline{\zeta}]$.

Let ${\lambda_u}_i^g, {\lambda_s}_i^g$ be the costates.
The Hamiltonian equals
\begin{align}
\begin{aligned}
H &= 1 + \sum_i^{n_c} \sum_{g}^{n_g} {\lambda_u}_i^g \left( \alpha_i^g  \overline{R}_i^g - \beta_i^g u_i^g(t)\right) \\
&\quad +\sum_i^{n_c} \sum_{g}^{n_g} {\lambda_s}_i^g \left( \beta_i^g u_i^g(t) 
-\gamma_i^g s^g_i(t) 
+\frac{c}{n_c} \sum_{j=1}^{n_c} A_{ij}\left( s^g_j(t) - s^g_i(t) \right) \right),
\end{aligned}
\end{align}
with $H(T) = 0$.
Applying the maximum principle, we have
\begin{align}
\begin{aligned}
\frac{\mathrm{d} {\lambda}_{u_i}^g}{\mathrm{d} t} =& - \frac{\partial H}{\partial u_i^g}
=  \beta_i^g {\lambda_u}_i^g - \beta_i^g {\lambda_s}_{i}^g \\
\frac{\mathrm{d} {\lambda}_{s_i}^g}{\mathrm{d} t} =& - \frac{\partial H}{\partial s_i^g}
= -  {\lambda_u}_i^g \alpha_i^g 
\frac{\partial \overline{R}_i^g}{\partial s_i^g} 
- \sum_{r \neq g} {\lambda_u}_i^r \alpha_i^r \frac{\partial \overline{R}_i^r}{\partial s_i^g}
+  {\lambda_s}_i^g \left(1 + \frac{c}{n_c} \sum_{j=1} A_{ij}\right)
- \sum_{j\neq i} {\lambda_s}_j^g   \frac{c}{n_c}a_{ji}.
\end{aligned}
\end{align}
Let $\overline{\Psi}(t,{\lambda_u}) := \sum_i^{n_c} \sum_{g}^{n_g} {\lambda_u}_i^g  \alpha_i^g \ \frac{ \delta_i W_{gq}^{+} s_i^q(t)}{\kappa+ \sum_{p=1}^{n_g} W_{gp}^{-} s_i^p(t)}.$ Then, the optimal controller $z^q_\star$ satisfies the following bang-bang conditions:
\begin{align}
\begin{aligned}
z^q_\star(t)
= \argmin_{z^q} H 
=& \argmin_{z^g} \ z^q (t) \overline{\Psi}(t,{\lambda_u}) \\
=& \begin{cases} 
1 & \text{if} \ \overline{\Psi}(t,{\lambda_u}) <0, \\ 
0 & \text{if} \ \overline{\Psi}(t,{\lambda_u}) >0, \\ 
\text{Undecided}  & \text{if}  \ \overline{\Psi}(t,{\lambda_u}) = 0.
\end{cases}
\end{aligned}
\end{align}

\section{Numerical Experiments}

\subsection{Baseline Intervention}
\label{sec.GRN.experiments}

We begin by presenting simulation results for two examples of small GRNs. In both cases, the topology is fixed, and the transcription rate of a particular gene is set to $0$ at a specific time. The first network comprises three genes, as illustrated in Figure~\ref{fig.GRN3.graph}. 
At time $t = 2$, the transcription rate $\alpha^2$ of gene $g = 2$ is set to $0$, while the rates $\alpha^g$ for $g = 1,3$ remain the same. Figure~\ref{fig.GRN-3} compares the network behavior without intervention (top panels) with that involving intervention (bottom panels), for both $u^g$ and $s^g$. As shown in Figures~\ref{fig.GRN3.Ub},~\ref{fig.GRN3.Sb}, in the untreated case, the network converges to a baseline steady state defined by its intrinsic regulatory feedbacks. Under the intervention that forces $\alpha^2$ to $0$, fewer (un)spliced RNAs $u^1$, $s^1$ are produced, and we observe an increase in $u^1$, and $s^1$ due to the silencing of gene $2$ and its repressive effects. By contrast, as a consequence of the lack of promotive effect of gene $2$ on gene $3$, $u^3$ and $ s^3$ decrease. The trajectories also shift toward a new equilibrium. This highlights how intervention changes not only the equilibrium but also the transient behavior of the system.

\begin{figure}[ht]
\centering
\subfigure[Network topology]{\includegraphics[width=0.27\textwidth]{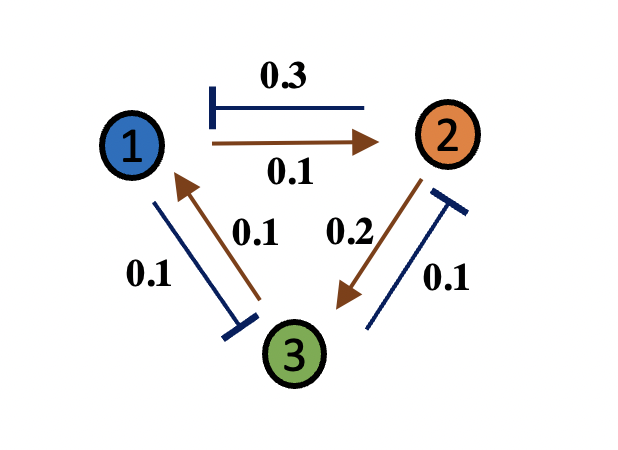}\label{fig.GRN3.graph}} 
\subfigure[Plot of $u^g(t)$ without intervention]{\includegraphics[width=0.35\textwidth]{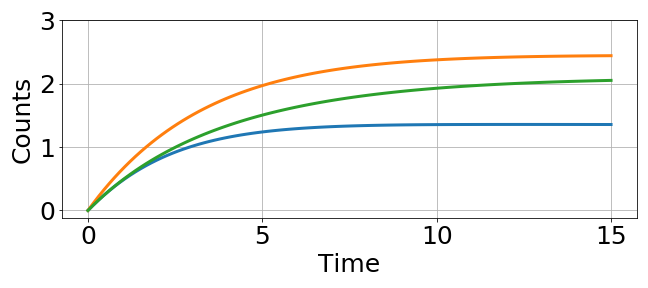}\label{fig.GRN3.Ub}} 
\subfigure[Plot of $s^g(t)$ without intervention]{\includegraphics[width=0.35\textwidth]{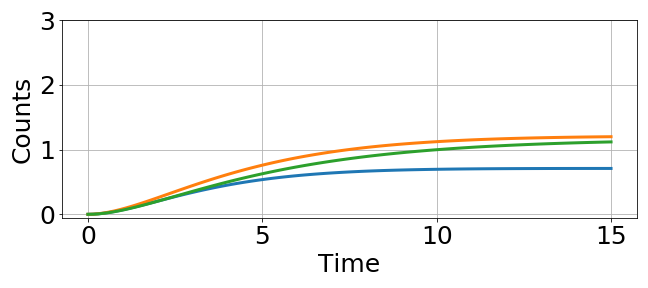}\label{fig.GRN3.Sb}}  
\subfigure[Plot of $u^g(t)$ with intervention]{\includegraphics[width=0.35\textwidth]{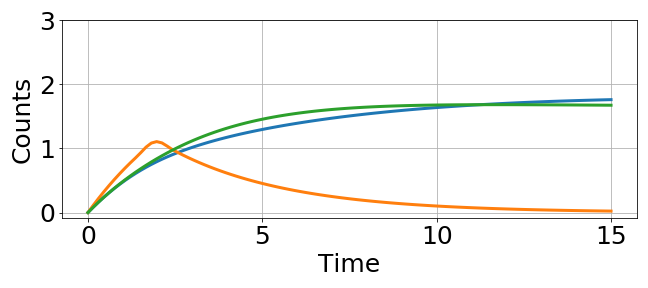}\label{fig.GRN3.U}} 
\subfigure[Plot of $s^g(t)$ with intervention]{\includegraphics[width=0.35\textwidth]{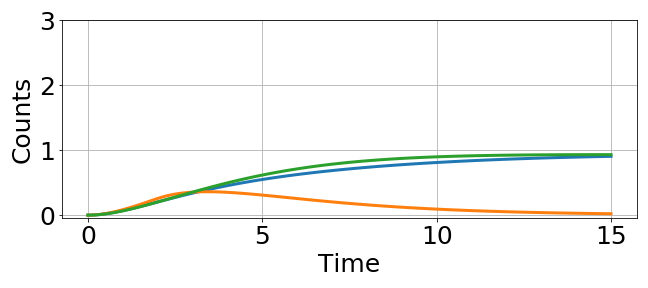}\label{fig.GRN3.S}}  \\
\caption{(a) Network topology of an example GRN. The red lines represent the effect of activators, while the blue lines represent the effect of repressors. (b)-(e)  Plots of $u^g$ and $s^g$ with and without intervention for gene 1 (\protect\solidblueline), gene 2 (\protect\solidorangeline), gene 3 (\protect\solidgreenline).}
\label{fig.GRN-3}
\end{figure}

In the second example, we consider a GRN that consists of 5 genes as shown in Figure \ref{fig.GRN5.graph}. Similar to the previous case, at time $t = 2$, the transcription rate $\alpha^1$ of gene $g = 1$ is set to $0$, while the rates $\alpha^g$ for $g = 2,3,4,5$ remain the same. With the removal of the positive regulatory effect of gene $1$ on gene $5$, we observe a decrease in $u^5$, $s^5$. On the other hand, since gene $1$ acts as an inhibitor of gene $3$, after the intervention, $u^3$ and $s^3$ increase. We also remark that this scenario reflects mixed network dynamics, in which only a subset of the genes attains the equilibrium.

\begin{figure}[ht]
\centering
\subfigure[Network topology]{\includegraphics[width=0.27\textwidth]{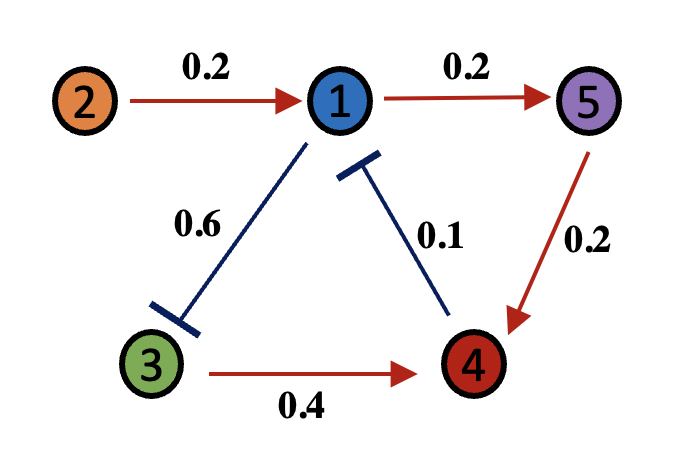}\label{fig.GRN5.graph}} 
\subfigure[Plot of $u^g(t)$ without intervention]{\includegraphics[width=0.35\textwidth]{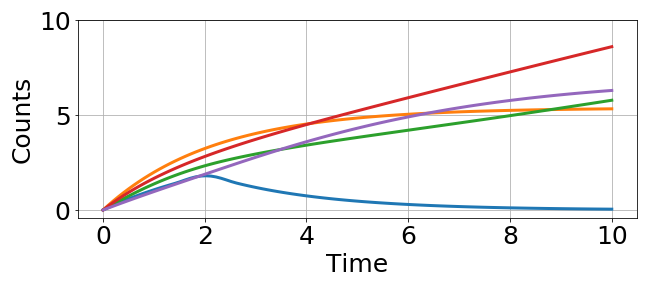}\label{fig.GRN5.U}} 
\subfigure[Plot of $s^g(t)$ without intervention]{\includegraphics[width=0.35\textwidth]{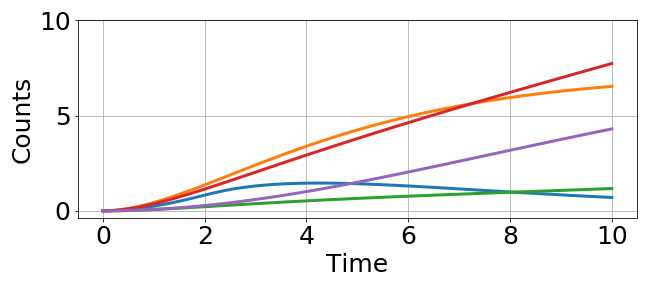}\label{fig.GRN5.S}} \\
\subfigure[Plot of $u^g(t)$ with intervention]{\includegraphics[width=0.35\textwidth]{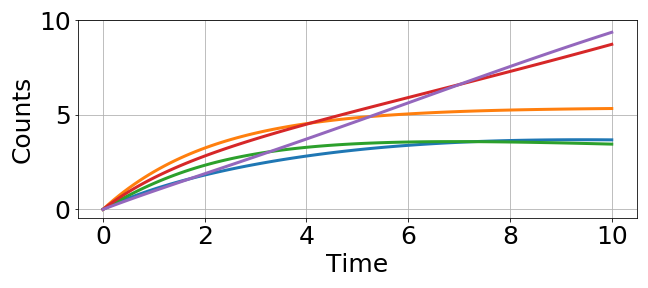}\label{fig.GRN5.Ub}} 
\subfigure[Plot of $s^g(t)$ with intervention]{\includegraphics[width=0.35\textwidth]{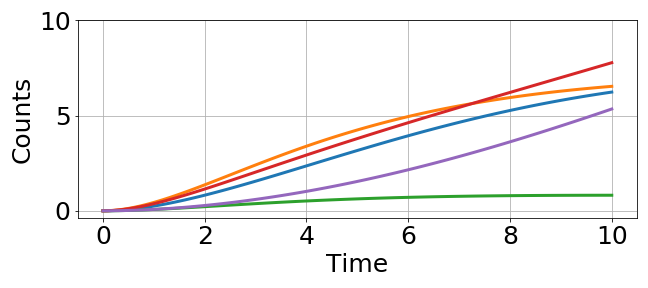}\label{fig.GRN5.Sb}} \\
\caption{(a) Network topology of an example GRN. The red lines represent the effect of the activators while the blue lines represent the effect of repressors.  (b)-(e)  Plots of $u^g$ and $s^g$ with and without intervention for gene 1 (\protect\solidblueline), gene 2 (\protect\solidorangeline), gene 3 (\protect\solidgreenline), gene 4 (\protect\solidredline), gene 5 (\protect\solidpurpleline).}
\label{fig.GRN-5}
\end{figure}

\subsection{Controlled GRN-Driven RNA Velocity.}
\label{sec.experiments.controlledGRN}

Next, following Section~\ref{sec.controlled.GRN}, we introduce control inputs that act on genes to simulate intervention. Let $z^q(t)\in [0,1]$.
To solve the time-optimal control problem~\eqref{eq.GRN.time-optimal-control}, we solve the associated two-point boundary value problem (TPBVP) numerically using the Forward-Backward-Sweep method (FBSM)~\cite{lenhart2007optimal}. Since our problem has free terminal time and partially fixed endpoints, we revise the standard FBSM that solves the basic variable-endpoint fixed time control problem as follows.
Starting with an initial guess for the optimal time $\hat{T}$, we solve the problem as a fixed-time problem. Specifically, we discretize the time interval $[0,\hat{T}]$ into equal-length bins $0 = t_0, t_1, \cdots, t_{N} = \hat{T}$. Then, we choose an initial guess for the optimal controller $z^q_0$ over the whole time interval. For the dynamics of states $u,s$, we start with the equilibrium of the uncontrolled system, i.e., we start from $u(0) = u^*$, $s(0) = s^*$, and solve this initial value problem forward in time according to~\eqref{eq.GRN.controlled}. 
Using the boundary condition ~\eqref{eq.GRN.boundary}, and values of $z^q_0, u, s$, we solve the costates dynamics backward in time according to Equation~\eqref{eq.GRN.costate}. 
To handle the costates $\lambda_s^r$ for $r \in \Iset$ whose terminal conditions are not specified, we either use an initial guess for these terminal conditions, or relax the original partially fixed-endpoint control Problem~\eqref{eq.GRN.time-optimal-control} to a variable-endpoint problem by introducing a terminal cost as a penalty function. We adopt the latter approach. Specially, we replace the constraints $s^r(T) = s_{\text{target}}^r$ for $r \in \Iset$ with a penalty function $\frac{\sigma}{2}\left\|s^r(T) - s_{\text{target}}^r \right\|_2^2$ as a terminal cost in Equation~\eqref{eq.GRN.time-optimal-control}. In this relaxed version, the PMP conditions remain the same except for the terminal conditions for the costate, which become
\begin{align}
\lambda_s^r(T) = \sigma \left(s^r(T) - s_{\text{target}}^r\right).
\end{align}
Using $z^q_0, u, s, \lambda_u,\lambda_s$, we update the controller $z^q_1$ via \eqref{eq.GRN.switch}. This process is repeated until a convergence condition is satisfied.


We next simulate the intervention problem over two GRNs and solve the time-optimal control problems. Similar to Section \ref{sec.GRN.experiments}, we perform two experiments on two GRNs that consist of $3$ and $5$ genes, respectively. In both examples, we modify the GNR graphs by including a self-loop as illustrated in  Figure \ref{fig.control3.graph} and Figure \ref{fig.control5.graph}. 
For the $3$-gene network, we control gene $2$ with the targeted gene being gene $3$ with $s^3_{\text{target}}=0.4$. For the $5$-gene newtork, we control gene $1$ and aim to target gene $5$ with $s^5_{\text{target}}=0.4$. All experiments are initialized at their respective equilibria of the uncontrolled dynamics. The solution of the time-optimal control problem indicates that we should set $z^q(t) \equiv 0$ in both cases. The optimal time to reach the targeted values is $2.32$s for the $3$-gene network and $9.38$s for the $5$-gene newtork.

\begin{figure}[ht]
\centering
\subfigure[GRN topology]{\includegraphics[width=0.3\textwidth]{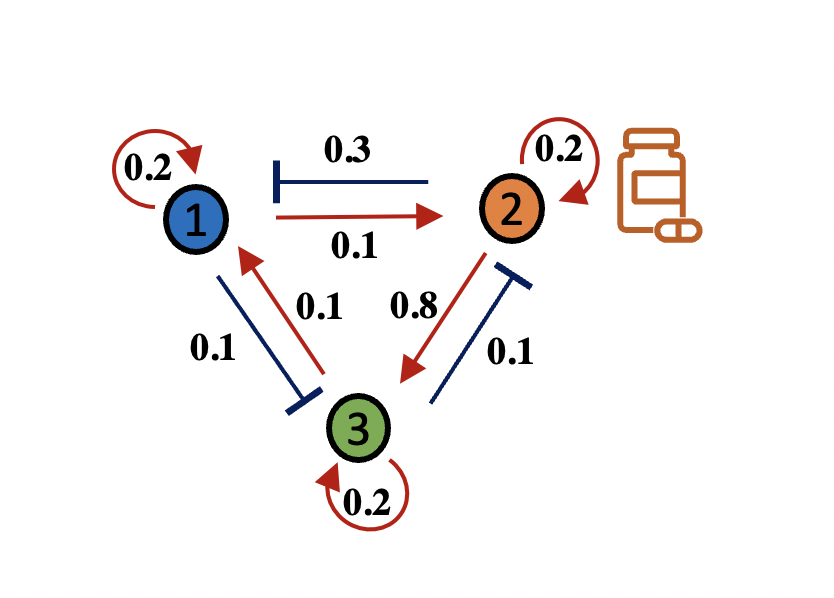}\label{fig.control3.graph}} 
\subfigure[Plot of $s^g(t)$]{\includegraphics[width=0.6\textwidth]{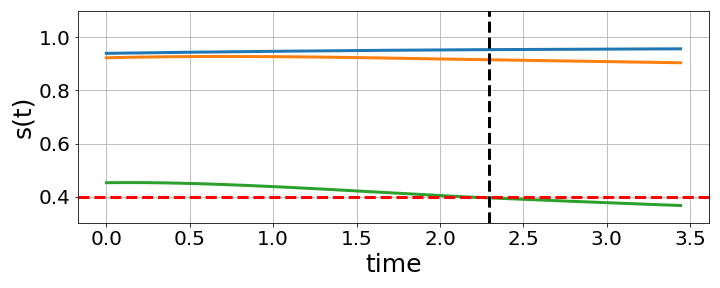}\label{fig.control3.S}} 
\caption{(a) Network topology of the GRN.  (b)  Plots of $s^g$ for gene 1 (\protect\solidblueline), gene 2 (\protect\solidorangeline), gene 3 (\protect\solidgreenline), targeted value of gene 3 (\protect\dottedredline), and the optimal time $T$ (\protect\dottedblackline)}
\label{fig.control.3}
\end{figure}

\begin{figure}[ht]
\centering
\subfigure[GRN topology]{\includegraphics[width=0.3\textwidth]{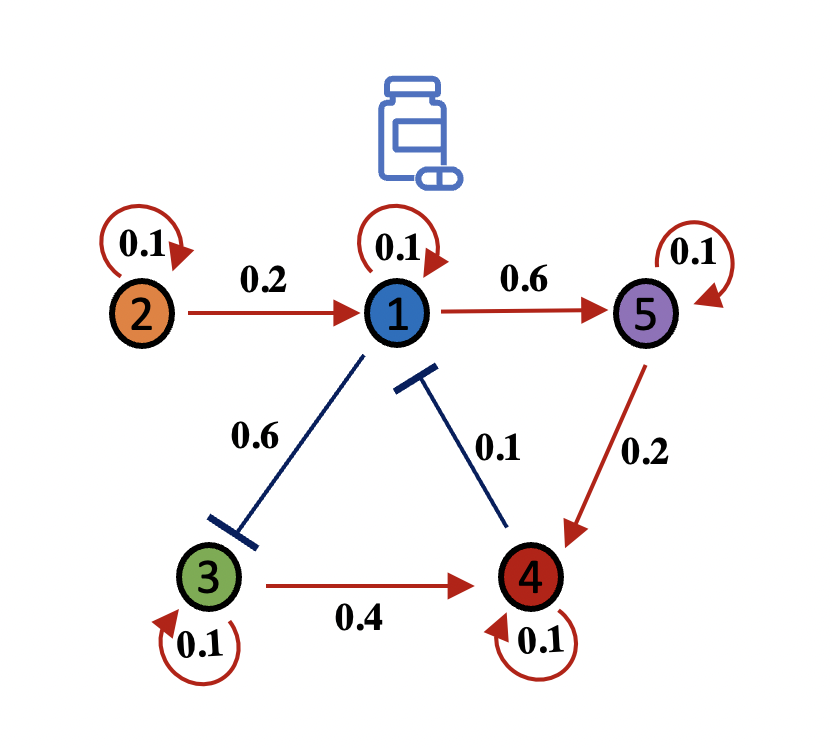}\label{fig.control5.graph}} 
\subfigure[Plot of $s^g(t)$]{\includegraphics[width=0.6\textwidth]{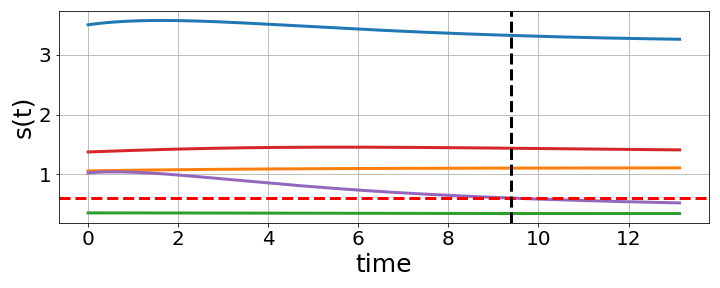}\label{fig.control5.S}}
\caption{(a) Network topology of the GRN.  (b) Plots of $s^g$ for gene 1 (\protect\solidblueline), gene 2 (\protect\solidorangeline), gene 3 (\protect\solidgreenline), gene 4 (\protect\solidredline), gene 5 (\protect\solidpurpleline), targeted value of gene 5 (\protect\dottedredline), and the optimal time $T$ (\protect\dottedblackline)}
\label{fig.control.5}
\end{figure}

\subsection{Controlled Spatially-Coupled GRN-Driven RNA Velocity}

Finally, we consider a network of 5 cells where each cell has the same GRN of 3 genes as considered in Section \ref{sec.experiments.controlledGRN}. Those $5$ cells form a complete graph, but which cell is affected by the intervention is random. In the first case, we assume all cells are affected by drug, i.e., $\delta_i = 1$ for all $i = 1,\cdots,5$. In the second case, we consider only cells $1,3,5$ are affected by the drug, i.e., $\delta_{i} = 1$ for $i=1,3,5$ and $\delta_{j} = 0$ for $j = 2,4$. All experiments are initialized at the equilibrium of the uncontrolled dynamics.
Let $z^q(t)\in [0,1]$, the time-optimal control solution shows that the optimal control is identically zero, i.e., $z^q(t) \equiv 0$, in both cases. In the first case, the optimal time $T^*$ for all cells to reach the targeted level of $s^3_{\text{target}}=0.4$ is $T^* = 2.32$s, which matches that of the single cell case as studied in Section \ref{sec.experiments.controlledGRN}. In the second case, only the cell $1,3$ reaches the targeted level $s^3_{\text{target}}=0.4$ with the $T^* = 2.80$s. For the uncontrolled cells, we witness a decrease in the level of $s^3$ due to the interconnection between cells.

\begin{figure}[ht]
\centering
\begin{minipage}{0.45\textwidth}
    \centering
    \subfigure[Cellular network]{
    \includegraphics[width=\textwidth]{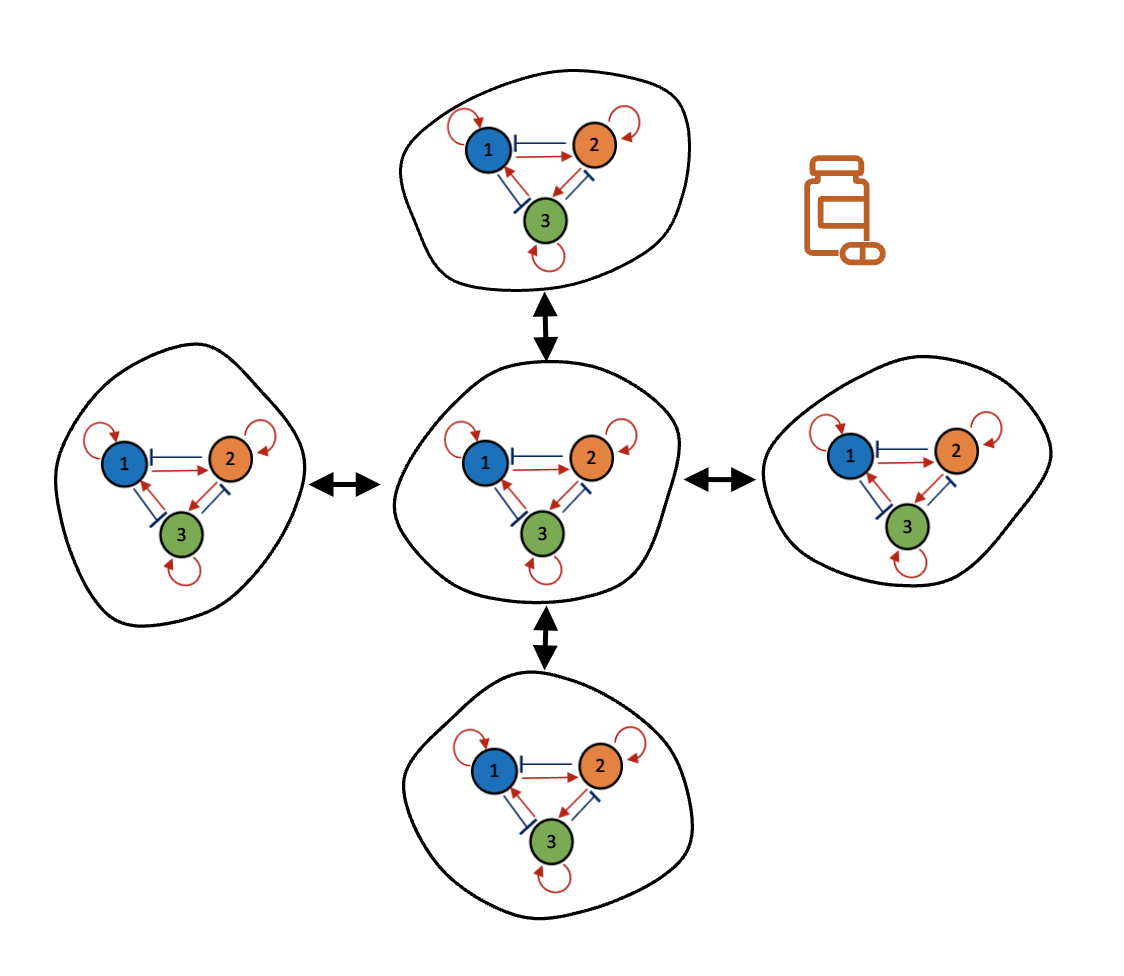}
    \label{fig.intercell.graph}}
\end{minipage}
\hfill
\begin{minipage}{0.5\textwidth}
    \centering
    \subfigure[Plot of $s^g(t)$ in cell 1 (controlled)]{
        \includegraphics[width=\textwidth]{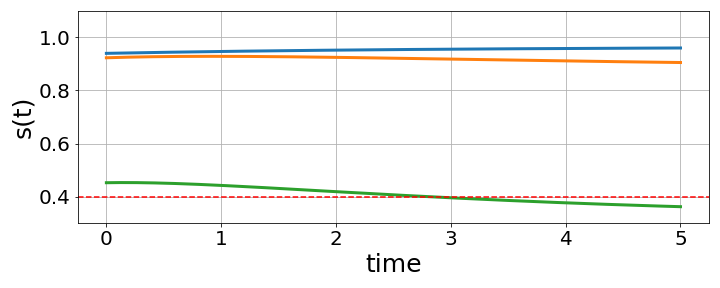}
        \label{fig.intercell.1}}
    \vspace{0.2em}
    \subfigure[Plot of $s^g(t)$ in cell 2 (uncontrolled)]{
        \includegraphics[width=\textwidth]{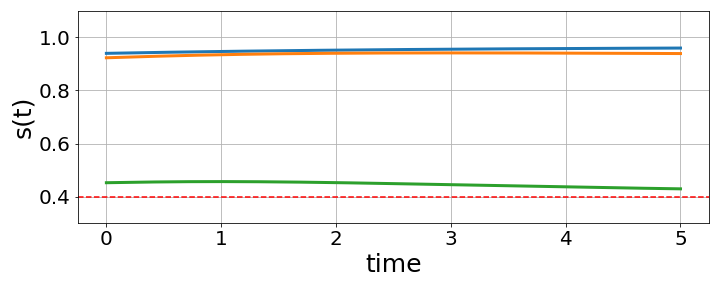}
        \label{fig.intercell.2}}
\end{minipage}
\caption{(a) A two-level network with 5 cells.  (b) (c): Plots of $s^g$ for gene 1 (\protect\solidblueline), gene 2 (\protect\solidorangeline), gene 3 (\protect\solidgreenline), and targeted value of gene 3 (\protect\dottedredline).}
\label{fig.intercellular}
\end{figure}

\section{Conclusion}
In this paper, we developed an RNA velocity framework that jointly models intracellular gene regulation and intercellular interactions. We analyzed the existence and stability of steady states in this two-level network, and established conditions for consensus across the cellular network. Building on this model, we formulated targeted drug intervention as a minimum-time optimal control problem, thus enabling principled and efficient transitions between cellular states of therapeutic interest. In future work, we will investigate learning the underlying parameters of the proposed ODE model from data by integrating multiomics datasets.

\section*{Acknowledgments and Funding Sources}
This work was supported by AbbVie Pharmaceuticals.

\newpage

\bibliographystyle{plain}
\bibliography{RNA}

@article{la2018rna,
  title={{RNA} velocity of single cells},
  author={La Manno, Gioele and Soldatov, Ruslan and Zeisel, Amit and Braun, Emelie and Hochgerner, Hannah and Petukhov, Viktor and Lidschreiber, Katja and Kastriti, Maria E and L{\"o}nnerberg, Peter and Furlan, Alessandro and others},
  journal={Nature},
  volume={560},
  number={7719},
  pages={494--498},
  year={2018},
  publisher={Nature Publishing Group UK London}
}

@article{santillan2008use,
  title={On the use of the {H}ill functions in mathematical models of gene regulatory networks},
  author={Santill{\'a}n, Moises},
  journal={Mathematical Modelling of Natural Phenomena},
  volume={3},
  number={2},
  pages={85--97},
  year={2008},
  publisher={EDP Sciences}
}

@article{armingol2020deciphering,
  title={Deciphering cell--cell interactions and communication from single-cell data},
  author={Armingol, Estefania and Officer, Adam and Harismendy, Olivier and Lewis, Nathan E},
  journal={Nature Reviews Genetics},
  volume={22},
  number={2},
  pages={71--88},
  year={2020},
  publisher={Nature Publishing Group}
}

@article{cang2023screening,
  title={Screening cell--cell communication in spatial transcriptomics with spatially resolved ligand--receptor interactions},
  author={Cang, Zixuan and Nie, Qing},
  journal={Nature Methods},
  volume={20},
  number={4},
  pages={558--567},
  year={2023},
  publisher={Nature Publishing Group}
}

@article{alon1986eigenvalues,
  title={Eigenvalues and expanders},
  author={Alon, Noga},
  journal={Combinatorica},
  volume={6},
  number={2},
  pages={83--96},
  year={1986},
  publisher={Springer}
}

@article{wu2002signaling,
  title={Signaling in plants by intercellular {RNA} and protein movement},
  author={Wu, Xuelin and Weigel, Detlef and Wigge, Philip A},
  journal={Genes \& development},
  volume={16},
  number={2},
  pages={151--158},
  year={2002},
  publisher={Cold Spring Harbor Lab}
}

@article{kehr2018long,
  title={Long distance {RNA} movement},
  author={Kehr, Julia and Kragler, Friedrich},
  journal={New Phytologist},
  volume={218},
  number={1},
  pages={29--40},
  year={2018},
  publisher={Wiley Online Library}
}

@article{bayraktar2017cell,
  title={Cell-to-cell communication: microRNAs as hormones},
  author={Bayraktar, Recep and Van Roosbroeck, Katrien and Calin, George A},
  journal={Molecular oncology},
  volume={11},
  number={12},
  pages={1673--1686},
  year={2017},
  publisher={Wiley Online Library}
}

@article{o2020rna,
  title={{RNA} delivery by extracellular vesicles in mammalian cells and its applications},
  author={O’Brien, Killian and Breyne, Koen and Ughetto, Stefano and Laurent, Louise C and Breakefield, Xandra O},
  journal={Nature reviews Molecular cell biology},
  volume={21},
  number={10},
  pages={585--606},
  year={2020},
  publisher={Nature Publishing Group UK London}
}

@article{sontag2003differential,
  title={For differential equations with $r$ parameters, $2 r+ 1$ experiments are enough for identification},
  author={Sontag, Eduardo},
  journal={Journal of Nonlinear Science},
  volume={12},
  number={6},
  pages={553--583},
  year={2003},
  publisher={Springer}
}

@book{conway1999sphere,
  title={Sphere Packings, Lattices and Groups},
  author={Conway, John H. and Sloane, Neil J. A.},
  edition={3rd},
  publisher={Springer},
  year={1999}}

@article{nils1994survey,
  title={A survey of {Ramanujan} graphs},
  author={Lubotzky, Alexander},
  journal={Discrete Mathematics},
  volume={138},
  number={1-3},
  pages={119--142},
  year={1995},
  publisher={Elsevier}
}

@article{hoory2006expander,
  title={Expander graphs and their applications},
  author={Hoory, Shlomo and Linial, Nathan and Wigderson, Avi},
  journal={Bulletin of the American Mathematical Society},
  volume={43},
  number={4},
  pages={439--561},
  year={2006}
}

@book{vinter2010optimal,
  title={Optimal control},
  author={Vinter, Richard B},
  year={2010},
  publisher={Springer}
}

@article{melman2018cauchy,
  title={Cauchy, {Gershgorin}, and matrix polynomials},
  author={Melman, Aaron},
  journal={Mathematics Magazine},
  volume={91},
  number={4},
  pages={274--285},
  year={2018},
  publisher={Taylor \& Francis}
}

@book{brauer2012mathematical,
  title={Mathematical Models in Population Biology and Epidemiology},
  author={Brauer, Fred and Castillo-Chavez, Carlos and Castillo-Chavez, Carlos},
  year={2012},
  publisher={Springer}
}

@article{srinivasan2022guide,
  title={A guide to the {Michaelis--Menten} equation: steady state and beyond},
  author={Srinivasan, Bharath},
  journal={The FEBS journal},
  volume={289},
  number={20},
  pages={6086--6098},
  year={2022},
  publisher={Wiley Online Library}
}

@book{alon2019introduction,
  title={An Introduction to Systems Biology: Design Principles of Biological Circuits},
  author={Alon, Uri},
  year={2019},
  publisher={Chapman and Hall/CRC}
}

@article{fang2020nonequilibrium,
  title={Nonequilibrium thermodynamics in cell biology: Extending equilibrium formalism to cover living systems},
  author={Fang, Xiaona and Wang, Jin},
  journal={Annual review of biophysics},
  volume={49},
  number={1},
  pages={227--246},
  year={2020},
  publisher={Annual Reviews}
}

@article{bergen2020generalizing,
  title={Generalizing {RNA} velocity to transient cell states through dynamical modeling},
  author={Bergen, Volker and Lange, Marius and Peidli, Stefan and Wolf, F Alexander and Theis, Fabian J},
  journal={Nature biotechnology},
  volume={38},
  number={12},
  pages={1408--1414},
  year={2020},
  publisher={Nature Publishing Group US New York}
}

@article{li2024tfvelo,
  title={TFvelo: gene regulation inspired RNA velocity estimation},
  author={Li, Jiachen and Pan, Xiaoyong and Yuan, Ye and Shen, Hong-Bin},
  journal={Nature Communications},
  volume={15},
  number={1},
  pages={1387},
  year={2024},
  publisher={Nature Publishing Group UK London}
}

@article{khammash2022cybergenetics,
  title={Cybergenetics: Theory and applications of genetic control systems},
  author={Khammash, Mustafa H},
  journal={Proceedings of the IEEE},
  volume={110},
  number={5},
  pages={631--658},
  year={2022},
  publisher={IEEE}
}

@article{kang2020graph,
  title={From graph topology to {ODE} models for gene regulatory networks},
  author={Kang, Xiaohan and Hajek, Bruce and Hanzawa, Yoshie},
  journal={Plos one},
  volume={15},
  number={6},
  pages={e0235070},
  year={2020},
  publisher={Public Library of Science}
}

@article{chen2025graphvelo,
  title={GraphVelo allows for accurate inference of multimodal velocities and molecular mechanisms for single cells},
  author={Chen, Yuhao and Zhang, Yan and Gan, Jiaqi and Ni, Ke and Chen, Ming and Bahar, Ivet and Xing, Jianhua},
  journal={bioRxiv},
  pages={2024--12},
  year={2025}
}

@book{lenhart2007optimal,
  title={Optimal control applied to biological models},
  author={Lenhart, Suzanne and Workman, John T},
  year={2007},
  publisher={Chapman and Hall/CRC}
}

@book{haddad2010nonnegative,
    author = {Haddad, Wassim M. and Chellaboina, {VijaySekhar} and Hui, Qing},
    title = {Nonnegative and Compartmental Dynamical Systems},
    publisher = {Princeton University Press},
    year = {2010}
}

@article{sepulchre2019feedback,
	author = {Sepulchre, R. and Drion, G. and Franci, A.},
	journal = {Annual Review of Control, Robotics, and Autonomous Systems},
	pages = {89-113},
	title = {Control Across Scales by Positive and Negative Feedback},
	volume = {2},
	year = {2019}}

@book{mesbashi2010graphs,
    author = {Mesbahi, Mehran and Egerstedt, Magnus},
    title = {Graph Theoretic Methods in Multiagent Networks},
    publisher = {Princeton University Press},
    year = {2010}
}

\newpage
\appendix
\centerline{\textbf{Supplementary Information}}

\section{Hill functions}
\label{app:Hill}

One of the most frequent models of gene regulation relies on \emph{Hill functions}~\cite{santillan2008use}. In the model, $x$ is used to denote the concentration of a transcription factor, using predefined units of concentration. The \emph{activation and repression Hill functions,} which are dimensionless, take values in $[0,1],$ and are rational functions of the form
\begin{align}
\Phi_{\text{act}}(x)=\frac{x^n}{\kappa^n+x^n}= \frac{1}{1+\left(\tfrac{\kappa}{x}\right)^n},\,\qquad
\, \Phi_{\text{rep}}(x)=\frac{\kappa^n}{\kappa^n+x^n}= \frac{1}{1+\left(\tfrac{x}{\kappa}\right)^n},
\end{align}
respectively. Here, $\kappa>0$ denotes what is known as the half-effective concentration which depends on the biological context but may be assumed to be constant for prespecified settings. The parameter $n>0$ is known as the Hill coefficient, and it controls the functional form of the Hill functions. For $n=1$, we have the so called hyperbolic (Michaelis–Menten) model, while for $n>1$ we have what is known as the positive cooperativity/ultrasensitive model. The case $0<n<1$ corresponds to the negative cooperativity/subsensitivity model. The Michaelis-Menten model is relevant in biological processes that are subject to saturation effects or systems in which a regulator can bind only to a small number of sites~\cite{srinivasan2022guide}. Our regulatory model also uses a Hill coefficient equal to $n=1,$ but decouples positive and negative influence factors by coupling them into the numerator and denominator, respectively, and enforcing them to be nonnegative. 

In gene-expression models, one typically uses $\Phi$ to describe the output rate of the transcriptional process via an affine transform
$$u(x)=u_{\min} + (u_{\max}-u_{\min})\, \Phi_{\text{act/rep}}(x),$$
where $u_{\min}$ and $u_{\max}$ control the dynamic range of molecular expression.
Deterministic models for regulated gene products $u$ integrate Hill functions into differential equations as follows:
\begin{align}
\frac{du}{dt} = \alpha\;\Phi(x) - \beta u,   
\end{align}
where $\alpha$ stands for the maximal transcription rate, $\beta$ for the  conversion/loss rate, and $\Phi(x)\in[0,1]$ is a Hill function or composition thereof. 
To facilitate a rigorous analysis of our model, we modify $\Phi(x)$ to involve only linear combinations of all regulatory units, and set $n=1$ (e.g., the Michaelis–Menten model). Although for gene regulation networks that have switch-like properties $n>1$ is preferred, the hyperbolic setting is more appropriate for capturing smooth agent responses. Furthermore, our production rates do no longer correspond to maxima/minima but rather basal values which allows all influence weights to be nonnegative (an assumption that facilitates analysis).

\section{Nonnegative dynamical systems}
\label{app:nonnegative}

In dynamical system models of biological systems, the unobserved and the observed state variables represent quantities like concentrations, abundances, etc.\ that take on nonnegative values. Hence, a key property of such systems is that, given a vector of nonnegative initial conditions at $t = 0$, all subsequent values of the state variables remain nonnegative for all $t \ge 0$. Such systems are referred to as \textit{nonnegative dynamical systems}. Here, we provide the necessary background on nonnegative dynamical systems, including structural characterization using the concept of \textit{essential nonnegativity}, as well as the definitions and criteria pertaining to equilibria and stability. A good reference is the book of Haddad et al.~\cite{haddad2010nonnegative}.

Consider a nonlinear system of the form
\begin{align}\label{eq:nonlinear_system}
\dot{x}(t) = f(x(t)),
\end{align}
where the state vector $x(t)$ takes values in $\Rset^n$. Let $\Rset^n_+$ denote the nonnegative orthant in $\Rset^n$, i.e., the set of all $x = [x_1, \dots, x_n]^\top \in \Rset^n$ such that $x_i \ge 0$ for all $i$. We say that the system \eqref{eq:nonlinear_system} is a nonnegative dynamical system if $\Rset^n_+$ is \textit{forward invariant} under the system dynamics, i.e., if 
\begin{align}
x(0) \in \Rset^n_+ \quad \Longrightarrow \quad x(t) \in \Rset^n_+, \, \forall t \ge 0.
\end{align}
\begin{Definition}Consider the system \eqref{eq:nonlinear_system}. The vector field $f$ is said to be \emph{essentially nonnegative} if, for any  $x \in \mathbb{R}_+^n$ and any index $i$ such that $x_i = 0$, it holds that $f_i(x) \geq 0$. 
\end{Definition}
The system \eqref{eq:nonlinear_system} is a nonnegative dynamical system if and only if the vector field $f$ is essentially nonnegative \cite[Proposition~2.1]{haddad2010nonnegative}.

Next, we discuss the notions of equilibria and stability for nonnegative dynamical systems. Suppose $x_e \in \Rset^n_+$ is an equilibrium of \eqref{eq:nonlinear_system}, i.e., $f(x_e) = 0$. We say that $x_e$ is \textit{stable in the sense of Lyapunov} with respect to $\Rset^n_+$ if, for all $\epsilon>0$, there exists a $\delta = \delta(\epsilon)>0$ such that $\vnorm{x(0) - x_e}_2 < \delta$ and $x(0) \in \Rset^n_+$ implies $\vnorm{x(t) - x_e}_2 < \epsilon$ and $x(t) \in \Rset^n_+$ for all $t > 0$. Furthermore, $x_e$ is \textit{asymptotically stable} with respect to $\Rset^n_+$ if it is stable in the sense of Lyapunov and if there exists a $\delta > 0$ such that $\vnorm{x(0) - x_e}_2 < \delta$ and $x(0) \in \Rset^n_+$ implies $x(t) \to x_e$ as $t \to \infty$. Finally, $x_e$ is \textit{globally asymptotically stable} with respect to $\Rset^n_+$ if it is stable in the sense of Lyapunov and if $x(t) \to x_e$ as $t \to \infty$ for any $x(0) \in \Rset^n_+$. In this latter case, $x_e$ is the unique equilibrium. The Lyapunov direct method \cite[Theorem~2.1]{haddad2010nonnegative} allows us to study the stability of the system without explicitly solving for the trajectories: Suppose there is a continuous differentiable function $V$ defined on an open set ${\cal D}$ containing $\Rset^n_+$, such that:
\begin{align}
\begin{split}
V(x_e) = 0, \qquad & \\
V(x) > 0, \qquad & x \in {\cal D}, \quad x \neq x_e \\
\dot{V}(x) := \frac{\partial V(x)}{\partial x}f(x) \leq 0, \qquad & x \in {\cal D}, \quad x \neq x_e.
\end{split}
\end{align}
Then the equilibrium point $x_e$ is Lyapunov stable w.r.t.\ $\Rset^n_+$. Suppose, further, that
\begin{align}\label{eq:Lyapunov_condition}
\dot{V}(x) < 0, \qquad x \in {\cal D}, \quad x \neq x_e.
\end{align}
Then $x_e$ is asymptotically stable w.r.t.\ $\Rset^n_+$. Finally, if \eqref{eq:Lyapunov_condition} holds and, in addition, $V(x) \to \infty$ whenever $\vnorm{x}_2 \to \infty$, then $x_e$ is globally asymptotically stable w.r.t.\ $\Rset^n_+$.

\section{Omitted proofs}
\subsection{Proof of Theorem~\ref{thm.ss}}
At a steady state, we must have
\begin{align}
\begin{aligned}
\alpha R(s^*) - \beta u^* = 0,\quad
\beta u^* - \gamma s^*=0. 
\end{aligned}
\end{align}
Hence, because $\gamma \succ 0$, $s^*$ satisfies
\begin{align}
 s^* = \gamma^{-1} \alpha R(s^*).
\end{align}
That is, $s^*$ is a fixed point of $F(\cdot):=\gamma^{-1} \alpha R(\cdot)$. We next use Brouwer's fixed-point theorem to establish conditions for the existence of at least one fixed point $s^* \in \Rset^{n_g}_+$. 

By the definition of $R$, the function $F$ is continuous. Next, consider a closed box in $\Rset^{n_g}_+$ defined as 
$\Bcal:=\left\{  s \in \Rset^{n_g}: 0 \leq s \leq M\right\}$ for some $M \in \Rset^{n_g}$ where the inequality is coordinatewise and $0 \leq M_k<\infty$ for all $k = 1,\cdots, n_g$.
Note that $\Bcal$ is a compact and convex subset of $\Rset^{n_g}$. By Brouwer's fixed-point theorem, if $F(\Bcal) \subseteq \Bcal$, then $F$ has at least one fixed point $s^* \in \Bcal$. We therefore have to find conditions under which $F(\Bcal) \subseteq \Bcal$. To this end, notice that since $W^{+}$ and $W^{-}$ are nonnegative matrices,  $R_g(s)$ can be bounded as follows:
\begin{align}
0\leq R_g(s)
= \frac{\kappa+ [W^{+} s]_g}{\kappa+  [W^{-} s]_g}
\leq \frac{\kappa+ [W^{+} s]_g}{\kappa}
\end{align}
Thus, we can bound the vectors $R(s)$ via $R(s) \leq \bm{1}+ \frac{1}{\kappa} W^{+} s$, and $F(s)$ as $F(s) \leq \gamma^{-1} \alpha \left( \bm{1}+ \frac{1}{\kappa}W^{+} s\right)$, where $\bm{1} := [1,\dots,1]^\top$ and the inequalities are coordinatewise.
Denote $\Lambda = \frac{1}{\kappa} \gamma^{-1} \alpha W^{+}$.
Since $\gamma,\alpha,W^{+}$ have nonnegative entries, the linear map $s \to \Lambda s + \gamma^{-1} \alpha  \bm{1}$ is monotone.
When $\rho\left( \Lambda \right) <1$, $(I-\Lambda)^{-1}$ exists, and we can define $M := (I-\Lambda)^{-1} \gamma^{-1} \alpha  \bm{1}$. Then, for $0\leq s \leq M $, we have
\begin{align}
F(s) \leq \Lambda s + \gamma^{-1} \alpha  \bm{1} 
\stackrel{(a)}{\leq}\Lambda M  + \gamma^{-1} \alpha  \bm{1} 
\stackrel{(b)}{\leq} \Lambda M  + (I-\Lambda) M
= M,
\end{align}
where (a) follows from the monotonicity of the linear map $\Lambda$, and (b) holds since by definition of $M$, we have
$\gamma^{-1} \alpha  \bm{1} = (I-\Lambda)M$.

We thus conclude that, when $\rho\left( \Lambda \right) <1$, $F$ has at least one fixed point $s^*$. In this case, we also have $u^* = \beta^{-1} \gamma s^*$.

\subsection{Proof of Lemma~\ref{lemma.stability0}}
When $W^{-}$ is a zero matrix, we have a linear system
\begin{align}
\begin{aligned}
\frac{d u}{dt} = \alpha \left(\kappa \bm{1} + W^{+} s \right) - \beta u,\quad
\frac{d s}{dt} = \beta u - \gamma s,
\end{aligned}
\label{eq.GRN.vec.linear}
\end{align}
where $\bm{1}$, as before, stands for the all-ones vector. Setting the right-hand side to $0$, we arrive at
\begin{align}
s^* = \kappa \left( \gamma - \alpha W^{+} \right)^{-1} \alpha \bm{1}, \quad u^* = \beta^{-1} \gamma s^*.
\end{align}

To examine the stability of $(u^*,s^*)$, notice that \eqref{eq.GRN.vec.linear} can be written as
\begin{align}
\begin{aligned}
\frac{d }{dt} \begin{bmatrix}
u\\
s
\end{bmatrix}  
= \underbrace{\begin{bmatrix} -\beta & \alpha W^{+} \\
\beta & -\gamma \end{bmatrix} }_{=: P}
\begin{bmatrix}u\\s\end{bmatrix} 
+\begin{bmatrix} \kappa \alpha \bm{1} \\0\end{bmatrix} .
\end{aligned}
\end{align}

For the linear system to be stable, $P$ must be a Hurwitz matrix. By the Gershgorin disk theorem~\cite{melman2018cauchy}, a sufficient condition to guarantee that all eigenvalues of $P$ have negative real parts (or, equivalently, all eigenvalues of $-P$ have positive real parts) requires that
\begin{align}
\beta^g > \sum_{h} \left|\left[\alpha W^{+} \right]_{gh}\right| , \quad \text{and} \quad 
|\gamma^g| > \sum_{h} \left|\beta_{gh}\right| ,\quad \forall g = 1,\cdots, n_g.
\end{align}
Since $W^{+}$ is nonnegative, and $\alpha$, $\beta$, $\gamma$ are diagonal matrices with positive entries, we have
\begin{align}
\beta^g > \alpha^g \sum_{h} W^{+}_{gh} , \quad \text{and} \quad
\gamma^g >\beta^g ,\quad \forall g = 1,\cdots, n_g.
\end{align}
This completes the proof. 

\subsection{Proof of Theorem~\ref{theorem.stablity}}
We first present a lemma that will be useful in subsequent proofs.
\begin{lemma}
If there exists a $\delta >0$ such that $W^{-}$ satisfies $\min_g [W^{-}s]_g \geq \delta \vnorm{s}_1$ for all $s \ge 0$, then $R(s)$ is globally Lipschitz on the nonnegative orthant, i.e., there exists $ 0 \leq \omega < \infty$ such that, for all $s,s' \ge 0$,
\begin{align}
\vnorm{R(s) - R(s')}_2 \leq \omega \vnorm{s - s'}_2.
\label{eq.Lipschitz}
\end{align}
\end{lemma}

\begin{proof}
We first compute the Jacobian $J_R$ of $R$ w.r.t $s$. For $g,h = 1,\cdots, n_g$, 
\begin{align}
\begin{aligned}
[J_R]_{gh}&=\frac{\partial R_g(s)}{\partial s^h}\\
&= \frac{W_{gh}^{+} \left(\kappa+ \sum_{q=1}^{n_g} W_{gq}^{-} s^q\right) 
- W_{gh}^{-} \left(\kappa+ \sum_{q=1}^{n_g} W_{gq}^{+} s^q\right)}{\left(\kappa+ \sum_{q=1}^{n_g} W_{gq}^{-} s^q \right)^2} \\
&= \frac{ \kappa \left(  W_{gh}^{+}  -  W_{gh}^{-}\right)
+ \sum_{q=1}^{n_g} \left(W_{gq}^{-} W_{gh}^{+} 
- W_{gh}^{-}  W_{gq}^{+}\right) s^q}{\left(\kappa+ \sum_{q=1}^{n_g} W_{gq}^{-} s^q \right)^2}.
\end{aligned}
\label{eq.J.R}
\end{align}

Let $c_1:= \max_{g,h} | W_{gh}^{+}  -  W_{gh}^{-} |$, 
$c_2:= \max_{g,h,q} |W_{gq}^{-} W_{gh}^{+} 
- W_{gh}^{-}  W_{gq}^{+} |$.
Since $W_{gh}^{+}$ and $W_{gh}^{-}$ are nonnegative and cannot be simultaneously positive, we have $c_1:= \max_{g,h} ( W_{gh}^{+},  W_{gh}^{-} )$, and $c_2:= \max_{g,h,q} ( W_{gq}^{-} W_{gh}^{+} 
,W_{gh}^{-}  W_{gq}^{+}) \leq c_1^2$.
We then have
\begin{align}
\begin{aligned}
\left|\frac{\partial R_g(s)}{\partial s^h}\right|
\leq  \frac{ \kappa c_1 +c_2  \Big(\sum_{q=1}^{n_g} s^q \Big)}{\left(\kappa+ \sum_{q=1}^{n_g} W_{gq}^{-} s^q \right)^2}
\leq  \frac{ \kappa c_1 + c_1^2  \vnorm{s}_1}{\left(\kappa+ \delta \vnorm{s}_1 \right)^2}.
\end{aligned}
\end{align}
Note that, for $r \geq 0$, the maximum of the function $f(r) := \frac{ \kappa c_1 +c_1^2  r}{\left(\kappa+ \delta r \right)^2}$ is attained at $r^* = \frac{\kappa \left( c_1 - 2 \delta \right)}{\delta c_1}$ if $c_1 > 2 \delta $ and $r^* = 0$ if $c_1 \leq 2 \delta$, with $f(r^*) = \frac{c_1}{\kappa}$ if $c_1 < 2 \delta $ and $f(r^*) = \frac{c_1^3}{4 \delta \kappa (c_1 - \delta)}$ otherwise. Hence, the Lipschitz constant is upper-bounded by
\begin{align}
\vnorm{J_R}_2 \leq \vnorm{J_R}_F 
=  \sqrt{\sum_{g,h}^{n_g} \left|\frac{\partial R_g(s)}{\partial s^h}\right|^2 } 
\leq n_g \max\left(\frac{c_1}{\kappa}, \frac{c_1^3}{4 \delta \kappa (c_1 - \delta)}\right) = : \omega,
\end{align}
where the second inequality holds since the spectral norm of the Jacobian is upper-bounded by its Frobenius norm. 
\end{proof}
We are now ready to present the proof of Theorem~\ref{theorem.stablity}. 
By definition, $V$ is positive definite and radially unbounded, i.e., $V(u,s) \to + \infty$ as $\vnorm{u}_2 \to +\infty$ and $\vnorm{s}_2 \to +\infty$. We next compute $\dot{V}(u,s)$ as follows. From \eqref{eq.GRN}, we have
\begin{align}
\begin{aligned}
\dot{V}(u,s)=& \left( u - u^*\right)^\top \dot{u} 
+ \left(s - s^*\right)^\top \dot{s}\\
=&\left( u - u^*\right)^\top \left(\alpha R(s) - \beta u\right) 
+ \left(s - s^*\right)^\top \left(\beta u - \gamma s\right). 
\end{aligned}    
\label{eq.Lyapunov.1}
\end{align}
Recall that the steady-state $(u^*,s^*)$ satisfies
\begin{align}
\begin{aligned}
\alpha R(s^*) - \beta u^* = 0,\quad
\beta u^* - \gamma s^*=0. 
\end{aligned}
\end{align}
Plugging these expressions into Equation~\eqref{eq.Lyapunov.1}, we get
\begin{align}
\begin{aligned}
\dot{V}(u,s)=& \left( u - u^*\right)^\top \left(\alpha R(s) - \beta u 
- \alpha R(s^*) + \beta u^* \right) 
+ \left(s - s^*\right)^\top \left(\beta u - \gamma s -\beta u^* + \gamma s^*\right) \\
=&\left(u - u^*\right)^\top \left[\alpha \left(R(s) - R(s^*) \right) - \beta \left( u - u^* \right) \right]
+ \left(s - s^*\right)^\top \left[\beta \left(u - u^*\right) - \gamma \left(s - s^*\right) \right]\\
=& \left(u - u^*\right)^\top \alpha \left(R(s) - R(s^*) \right) 
- \left(u - u^*\right)^\top \beta \left( u - u^* \right) 
+ \left(s - s^*\right)^\top \beta \left(u - u^*\right) 
- \left(s - s^*\right)^\top\gamma \left(s - s^*\right)\\
=& 
\left(u - u^*\right)^\top \alpha \left(R(s) - R(s^*) \right)  
+\begin{bmatrix} u - u^* \\ s - s^* \end{bmatrix}^\top 
\begin{bmatrix} -\beta & \frac{1}{2} \beta \\ \frac{1}{2} \beta & -\gamma \end{bmatrix} 
\begin{bmatrix} u - u^* \\ s - s^* \end{bmatrix}.
\end{aligned}    
\label{eq.Lyapunov.2}
\end{align}

We first focus on the first term $\left(u - u^*\right)^\top \alpha \left(R(s) - R(s^*) \right)  \in \Rset$. Since $R(\cdot)$ is uniformly Lipschitz on $\Rset^{n_g}_+$, for all $u,s \in \Rset^{n_g}_+$ we have
\begin{align}
\begin{aligned}
\left(u - u^*\right)^\top \alpha \left(R(s) - R(s^*) \right) 
\leq& \left|\left(u - u^*\right)^\top \alpha \left(R(s) - R(s^*) \right) \right| \\
\stackrel{(a)}{\leq} & \epsilon \vnorm{u - u^*}_2^2 
+ \frac{1}{4 \epsilon} \vnorm{\alpha \left(R(s) - R(s^*) \right)}^2\\
\stackrel{(b)}{\leq} &  \epsilon \vnorm{u - u^*}_2^2
+ \frac{1}{4 \epsilon} \vnorm{\alpha}_\op^2 \vnorm{R(s) - R(s^*)}_2^2 \\
\stackrel{(c)}{\leq} &  \epsilon \vnorm{u - u^*}_2^2
+ \frac{\omega^2}{4 \epsilon} \vnorm{\alpha}_\op^2 \vnorm{s - s^*}_2^2.
\end{aligned}
\end{align}
In (a), we made use of Young's inequality for real numbers with $\epsilon_0>0$, which asserts that for $a,b \in \Rset$, $ab \leq \frac{a^2}{2 \epsilon_0} + \frac{\epsilon_0 b^2}{2}$. Setting $\epsilon = \frac{1}{2 \epsilon_0}$, $a = \vnorm{u - u^*}_2$ and $b = \vnorm{\alpha \left(R(s) - R(s^*) \right)}_2$ gives (a).
Line (b) follows from the submultiplicative property, and line (c) holds due to~\eqref{eq.Lipschitz}. 
Plugging this into \eqref{eq.Lyapunov.2} we get
\begin{align}
\begin{aligned}
\dot{V}(u,s)
\leq&  
\begin{bmatrix} u - u^* \\ s - s^* \end{bmatrix}^\top 
\begin{bmatrix} -\beta + \epsilon I & \frac{1}{2} \beta \\
\frac{1}{2} \beta & -\gamma +\frac{\omega^2 }{4 \epsilon} \vnorm{\alpha}^2_\op I
\end{bmatrix} 
\begin{bmatrix} u - u^* \\ s - s^* \end{bmatrix},
\end{aligned}    
\end{align}
where $I$ is the identity matrix of dimension $n_g \times n_g$. Now, define $Q \in \Rset^{2n_g \times 2 n_g}$ as
\begin{align}
Q:= \begin{bmatrix} \beta - \epsilon I & -\frac{1}{2} \beta \\
- \frac{1}{2} \beta & \gamma - \frac{\omega^2 }{4 \epsilon} \vnorm{\alpha}^2_\op I
\end{bmatrix}. 
\end{align}
For $\dot{V} \leq - \begin{bmatrix} u - u^* \\ s - s^* \end{bmatrix}^\top  Q \begin{bmatrix} u - u^* \\ s - s^* \end{bmatrix} < 0$ to hold, we need $Q$ to be positive definite. Since $Q$ is symmetric, $Q$ is positive definite if and only if  
$\beta - \epsilon I$, the top left sumbatrix in $Q$, and its Schur complement in the block-matrix $Q$ defined below are positive definite, i.e.,
\begin{align}
\gamma - \frac{\omega^2 }{4 \epsilon} \vnorm{\alpha}_\op^2 I
- \frac{1}{4} \beta \left( \beta - \epsilon I\right)^{-1} \beta \succ 0.
\label{eq.Q.schur}
\end{align}
As $\gamma,\alpha,\beta$ are diagonal matrices, the above condition can be written as
\begin{align}
\beta^g > \epsilon, \quad
\gamma^g > \frac{\omega^2 }{4 \epsilon} \vnorm{\alpha}_\op^2 
+ \frac{{\beta^g}^2}{4(\beta^g - \epsilon )},
\quad \forall g = 1,\cdots, n_g.
\end{align}
Choose $\epsilon = \frac{\omega \vnorm{\alpha}_\op}{2}$ so that $\epsilon = \frac{\omega^2 }{4 \epsilon} \vnorm{\alpha}_\op^2$. Then we have
\begin{align}
\beta^g > \frac{\omega \vnorm{\alpha}_\op}{2}, \quad
\gamma^g > \frac{\omega \vnorm{\alpha}_\op}{2} 
+ \frac{{\beta^g}^2}{4(\beta^g - \frac{\omega \vnorm{\alpha}_\op}{2} )},
\quad \forall g = 1,\cdots, n_g.
\end{align}
By the Lyapunov theorem for nonnegative dynamical systems (see Appendix, Section~\ref{app:nonnegative}), the equilibrium point $(u^*,s^*)$ is globally asymptotically stable (and hence unique).

\subsection{Proof of Lemma~\ref{lemma.essentially.nonnegative}} Suppose $u^g,s^g$ have nonnegative coordinates.
First, consider the case $u_i^g = 0$. In that case, the right-hand side of $\frac{d u_i^g}{dt}$ only depends on $s_i^q$. Since the entries of $W^{\pm}$ are nonnegative, we have
\begin{align}
\frac{d u_i^g}{dt} = \alpha_i^g \cdot \frac{\kappa + \sum_q W_{gq}^{+} s_i^q}{\kappa + \sum_q W_{gq}^{-} s_i^q} \geq 0. 
\end{align}

Next, let $s_i^g = 0$, and assume $u_i^g \geq 0$ and $s_j^g \geq 0$ for all $g$ and  $j \neq i$. Then we have
\begin{align}
\frac{d s_i^g}{dt} = \beta_i^g u_i^g + c \sum_j A_{ij} s_j^g \geq 0.
\end{align}
Therefore, each component of the vector field is nonnegative when the corresponding state component is zero, and all others are nonnegative. 

\subsection{Proof of Theorem~\ref{thm.ss-multi}}
At equilibrium, we have
\begin{align}
\begin{aligned}
& \alpha_i^g \ \frac{\kappa+ \sum_{q=1}^{n_g} W_{gq}^{+} s_i^q}{\kappa+ \sum_{q=1}^{n_g} W_{gq}^{-} s_i^q}= \beta_i^g u_i^g \\
& \beta_i^g u_i^g 
= \gamma_i^g s^g_i(t) - {c} \sum_{j=1}^{n_c} A_{ij}\left( s^g_j - s^g_i \right) .
\end{aligned}
\end{align}
Hence, $s^i_g$ satisfy the following equations for all $i=1,\cdots,n_c$ and $g =1,\cdots,n_g$, 
\begin{align}
\alpha_i^g \ \frac{\kappa+ \sum_{q=1}^{n_g} W_{gq}^{+} s_i^q}{\kappa+ \sum_{q=1}^{n_g} W_{gq}^{-} s_i^q}
= \gamma_i^g s^g_i 
- {c} \sum_{j=1}^{n_c} A_{ij}\left( s^g_j - s^g_i \right).
\end{align}
We reuse the notation $R_i^g(s_i) = \frac{\kappa+ \sum_{q=1}^{n_g} W_{gq}^{+} s_i^q}{\kappa+ \sum_{q=1}^{n_g} W_{gq}^{-} s_i^q}$ so that the above condition can then be rewritten as
\begin{align}
\frac{\alpha_i^g}{\gamma_i^g} \ R_i^g(s_i)+ \frac{c}{\gamma_i^g} \sum_{j=1}^{n_c} A_{ij}  s_j^g
=  s^g_i +\left(\frac{c}{\gamma_i^g} \sum_{j=1}^{n_c} A_{ij}\right) s_i^g .
\end{align}

Define the mapping $F(s)= [F_1(s)^\top,\cdots,F_{n_c}(s)^\top]^\top \in \Rset^{n_g \cdot n_c}$ where each block component $F_i(s) = [F_i^1,\cdots,F_i^{n_g}]^\top \in \Rset^{n_g}$ is defined elementwise as
\begin{align}
    F_i^g\left(s\right) = \frac{\alpha_i^g \ R_i^g(s_i)+ {c} \sum_{j=1}^{n_c} A_{ij}  s_j^g}{\gamma_i^g +\left({c} \sum_{j=1}^{n_c} A_{ij}\right)},
\end{align}
for $g = 1,\cdots,n_g$, $i = 1,\cdots,n_c$.
Then, $s$ should be a fixed point of $F$. In addition, by the definition of $R$ in~\eqref{eq.R}, $F$ is continuous. Now consider a box in $\Rset^{n_g \cdot n_c}_+$ defined as $\Bcal:=\left\{  s \in \Rset^{n_g \cdot n_c}: 0 \leq s \leq M\right\}$ for some $M \in \Rset^{n_g \cdot n_c}_+$, where the inequality holds for each component/dimension and $0 \leq M_k<\infty$ for all $k = 1,\cdots, n_g \cdot n_c$. Such a closed box $\Bcal$ is a compact and convex subset of $\Rset^{n_g \cdot n_c}$ . By Brouwer's fixed-point theorem, if $F(\Bcal) \subseteq \Bcal$, then a fixed point $s^* = F(s^*)$ exists. To this end, notice that, since $W^{+}$ and $W^{-}$ are nonnegative matrices,  $R_i^g(s)$ can be bounded as

\begin{align}
0\leq R_i^g(s)
= \frac{\kappa+ [W^{+} s_i]_g}{\kappa+  [W^{-} s_i]_g}
\leq \frac{\kappa+ [W^{+} s_i]_g}{\kappa}.
\end{align}
Since the adjacency matrix $A$ of the cellular network is nonnegative, we have 
\begin{align}
F_i^g \leq \frac{\alpha_i^g \ R_i^g(s_i)+ {c} \sum_{j=1}^{n_c} A_{ij}  s_j^g(t)}{\gamma_i^g}    
\leq& \frac{\alpha_i^g}{\gamma_i^g}
+ \underbrace{\frac{1}{\kappa}\frac{\alpha_i^g}{\gamma_i^g} [W^{+} s_i]_g}_{=: T_1}
+ \underbrace{\frac{c}{\gamma_i^g} \sum_{j=1}^{n_c} A_{ij}  s_j^g(t)}_{=: T_2}.
\end{align}

Let $b = \text{vec}\left(\frac{\alpha_i^g}{\gamma_i^g}\right) \in \Rset^{n_g \cdot n_c}$. For a fixed cell $i$, $T_1$ captures the GRN inside the cell, while $T_2$ captures the spatial coupling among cells. Note that $T_1$ and $T_2$ in the above equations are linear in $s$, and we can define a block matrix $\Lambda \in \Rset^{n_g \cdot n_c \times n_g \cdot n_c}$ as follows. The diagonal elements of $\Lambda$ are set as $\Lambda_{ii}:= \frac{1}{\kappa} \gamma_i^{-1} \alpha_i W^{+} \in \Rset^{n_g  \times n_g}$ for $i=1,\cdots,n_c$. The off-diagonal elements of $\Lambda$ are set to $\Lambda_{ij} = {c} A_{ij} \gamma_i^{-1} I_{n_g}$ for $i,j=1,\cdots,n_c$, $i \neq j$. Thus, the block matrix $\Lambda$ equals 
\begin{align}
\Lambda = \text{diag}\left( \frac{1}{\kappa} \gamma_1^{-1} \alpha_1 W^{+}, \cdots, \frac{1}{\kappa} \gamma_{n_c}^{-1} \alpha_{n_c} W^{+}\right)    
+ {c} \left(A \otimes I_{n_g} \right) \text{diag}\left(\gamma_1^{-1},\cdots,\gamma_{n_c}^{-1}  \right).
\end{align}
We can therefore bound $F(s)$ according to $F(s) \leq b + \Lambda s$. Since $\Lambda$ is nonnegative, the linear map is monotone. When $\rho(\Lambda)<1$, $\left(I - \Lambda \right)^{-1}$ exists and we can define $M: = \left(I - \Lambda \right)^{-1} b$. Then, for $0 \leq s \leq M$, we have 
\begin{align}
F(s) \leq b + \Lambda s
\leq b + \Lambda M
= \left(I - \Lambda\right) M + \Lambda M
= M.
\end{align}
Hence, a fixed point in the nonnegative orthant exists by Brouwer's fixed-point theorem. 

\subsection{Proof of Lemma~\ref{lemma.stability0-multi}}
When $W^{-}$ is the zero-matrix, we have a linear system
\begin{align}
\begin{aligned}
\frac{d u_i^g}{dt} &= \frac{1}{\kappa} \alpha_i^g  \left(\kappa+ \sum_{q=1}^{n_g} W_{gq}^{+} s_i^q(t)\right)- \beta_i^g u_i^g(t),\\
\frac{d s_i^g}{dt} &= \beta_i^g u_i^g(t) 
-\gamma_i^g s^g_i(t) 
+ {c} \sum_{j=1}^{n_c} A_{ij}\left( s^g_j(t) - s^g_i(t) \right).
\end{aligned}
\label{eq.multi.linear}
\end{align}

Let $u^g = [u_1^g,u_2^g,\cdots u_{n_c}^g]^\top$.
To study the stability of the linear system, notice that for each $g$, \eqref{eq.multi.linear} can be written as
\begin{align}
\begin{aligned}
\frac{d u^g}{dt} &= \alpha^g 
+ \frac{1}{\kappa} \alpha^g \sum_p^{n_g} W_{gp}^{+} s^p - \beta^g u^g,\\
\frac{d s^g}{dt} &= \beta^g u^g
-\gamma^g s^g - {c} L s^g.
\end{aligned}
\end{align}
Let $u= [{u^1}^\top,\cdots {u^{n_g}}^\top]^\top \in \Rset^{n_g \cdot n_c}$, $L=\text{diag}(A \times \bm{1}_{n_c}) - A$ be the graph Laplacian, and let $\alpha$,$\beta$,$\gamma$ be block diagonal matrices, i.e., $\alpha= \text{diag}\left(\alpha^1,\cdots, \alpha^{n_g}\right)$, $\beta= \text{diag}\left(\beta^1, \cdots, \beta^{n_g}\right)$, $\gamma= \text{diag}\left(\gamma^1, \cdots, \gamma^{n_g}\right)$. Then, we have
\begin{align}
\begin{aligned}
\frac{d }{dt} \begin{bmatrix} u\\ s \end{bmatrix}  
= \underbrace{\begin{bmatrix} -\beta & \frac{1}{\kappa} \alpha \left( W^{+} \otimes I_{n_c} \right)  \\
\beta & -\gamma  -  {c} \left(I_{n_g} \otimes L \right)  \end{bmatrix} }_{=: P}
\begin{bmatrix}u\\s\end{bmatrix} 
+\begin{bmatrix}  \alpha \bm{1} \\0\end{bmatrix} .
\end{aligned}
\end{align}

For the linear system to be stable, $P$ must be Hurwitz. By the Gershgorin disk theorem~\cite{melman2018cauchy}, a sufficient condition to guarantee that all eigenvalues of $P$ have negative real parts (or, equivalently, all eigenvalues of $-P$ have positive real parts) is
\begin{align}
\begin{aligned}
\beta_i^g > \frac{1}{\kappa} \sum_{h}^{n_g} \left| \alpha_i  W_{gh}^{+}\right| , \quad
\left|\gamma_i^g  +  {c}  L_{ii}\right| > \beta_i^g +  {c} \sum_{j \neq i}^{n_c} |L_{ij}|.
\end{aligned}
\end{align}
Since $W^{+}$ is nonnegative, and $\alpha$, $\beta$, $\gamma$ are block diagonal matrices with positive entries, we have
\begin{align}
\begin{aligned}
\beta_i^g >  \frac{1}{\kappa}  \sum_{h}^{n_g} \alpha_i^g W_{gh}^{+} , \quad
\gamma_i^g +  {c}  L_{ii} > \beta_i^g + {c} \sum_{j \neq i}^{n_c} |L_{ij}|.
\end{aligned}
\end{align}
Since $L_{ii} = \sum_{j \neq i}^{n_c} |L_{ij}|$ by definition, we also have 
\begin{align}
\begin{aligned}
\beta_i^g > \frac{\alpha_i^g}{\kappa}  \sum_{h}^{n_g} W^{+}_{gh}, \quad
\gamma_i^g  > \beta_i^g .
\end{aligned}
\end{align}

\subsection{Proof of Theorem~\ref{theorem.stablity-multi}}
We first present a lemma on the continuity of the nonlinear function $R_i^g(s_i)$.
\begin{lemma}
Let $\vnorm{s_i}_1= \sum_{q=1}^{n_g} s_i^q$.
If for each cell $i \in \{ 1,\cdots, n_c \}$ there exists a $\delta_i >0$ such that $W^{-}$  satisfies $\min_g [W^{-}s]_g \geq \delta_i \vnorm{s_i}_1$ for all $s_i$ with nonnegative coordinates, then $R_i^g(s_i)$ is globally Lipschitz on the nonnegative orthant, i.e., there exists an $\omega_i \geq 0$ such that, for all $s_i, s'_i$ with nonnegative coordinates,
\begin{align}
\vnorm{R_i^g(s_i) - R_i^g(s'_i)}_2 \leq \omega_i \vnorm{s_i - s'_i}_2.
\end{align}
\end{lemma}

\begin{proof}
We first compute the Jacobian $J_i^g$ of $R_i^g$ w.r.t. $s_i$.  For $g,h = 1,\cdots, n_g$, 
\begin{align}
\begin{aligned}
[J_i^g]_{h}=&\frac{\partial R_i^g(s_i)}{\partial s_i^h}\\
=& \frac{W_{gh}^{+} \left(\kappa+ \sum_{q=1}^{n_g} W_{gq}^{-} s_i^q\right) 
- W_{gh}^{-} \left(\kappa+ \sum_{q=1}^{n_g} W_{gq}^{+} s_i^q\right)}{\left(\kappa+ \sum_{q=1}^{n_g} W_{gq}^{-} s_i^q \right)^2} \\
=& \frac{ \kappa \left(  W_{gh}^{+}  -  W_{gh}^{-}\right)
+ \sum_{q=1}^{n_g} \left(W_{gq}^{-} W_{gh}^{+} 
- W_{gh}^{-}  W_{gq}^{+}\right) s_i^q}{\left(\kappa+ \sum_{q=1}^{n_g} W_{gq}^{-} s_i^q \right)^2}.
\end{aligned}
\end{align}

Let $c_1:= \max_{g,h} \left( W_{gh}^{+},  W_{gh}^{-} \right)$. We then have
\begin{align}
\begin{aligned}
\frac{\partial R_i^g(s_i)}{\partial s_i^h}
\leq  \frac{ \kappa c_1 + c_1^2  \vnorm{s_i}_1}{\left(\kappa+ \delta_i \vnorm{s_i}_1 \right)^2}.
\end{aligned}
\end{align}
Hence, the Lipschitz constant is upper-bounded by
\begin{align}
\vnorm{J_i^g}_2 \leq 
\sqrt{\sum_{h}^{n_g} \left|\frac{\partial R_i(s_i)}{\partial s_i^h}\right|^2 } 
= \sqrt{n_g} \max\left(\frac{c_1}{\kappa}, \frac{c_1^3}{4 \delta_i \kappa (c_1 - \delta_i)}\right):=\omega_i.
\end{align}
This completes the proof. \end{proof}

By construction, $V$ is positive definite and radially unbounded. We next show that $\dot{V}(u,s) < 0$. From~\eqref{eq.multinetwork.dynamics}, we have
\begin{align}
\begin{aligned}
\dot{V}(u,s)=& \sum_{i=1}^{n_c} \sum_{g= 1}^{n_g} \left( \left( u_i^g - {u_i^g}^* \right) \dot{u}_i^g + \left( s_i^g - {s_i^g}^*\right) \dot{s}_i^g \right)\\
=& \sum_{i=1}^{n_c} \sum_{g= 1}^{n_g} \left( u_i^g - {u_i^g}^* \right) \left(\alpha_i^g \ \frac{\kappa+ \sum_{q=1}^{n_g} W_{gq}^{+} s_i^q}{\kappa+ \sum_{q=1}^{n_g} W_{gq}^{-} s_i^q}- \beta_i^g u_i^g(t)\right) \\
&+ \sum_{i=1}^{n_c} \sum_{g= 1}^{n_g} \left(s_i^g - {s_i^g}^*\right) \left(\beta_i^g u_i^g 
-\gamma_i^g s^g_i + {c} \sum_{j=1}^{n_c} A_{ij}\left( s^g_j - s^g_i \right) \right). 
\end{aligned}    
\label{eq.Lyapunov.multi.1}
\end{align}
Recall that the steady state $({u_i^g}^*,{s_i^g}^*)$ satisfies
\begin{align}
\begin{aligned}
\alpha_i^g \ \frac{\kappa+ \sum_{q=1}^{n_g} W_{gq}^{+} {s_i^q}^*}{\kappa+ \sum_{q=1}^{n_g} W_{gq}^{-} {s_i^q}^*}- \beta_i^g {u_i^g}^* = 0, \quad
\beta_i^g {u_i^g}^* 
-\gamma_i^g {s^g_i}^* + {c} \sum_{j=1}^{n_c} A_{ij}\left( {s^g_j}^* - {s^g_i}^* \right).
\end{aligned}
\end{align}
Plugging the previous equalities into~\eqref{eq.Lyapunov.multi.1}, we get

\begin{align}
\begin{aligned}
\dot{V}(u,s)=& \sum_{i=1}^{n_c} \sum_{g= 1}^{n_g} \left( u_i^g - {u_i^g}^* \right) 
\left[\alpha_i^g \ \frac{\kappa+ \sum_{q=1}^{n_g} W_{gq}^{+} s_i^q}{\kappa+ \sum_{q=1}^{n_g} W_{gq}^{-} s_i^q}- \beta_i^g u_i^g - \left( \alpha_i^g \ \frac{\kappa+ \sum_{q=1}^{n_g} W_{gq}^{+} {s_i^q}^*}{\kappa+ \sum_{q=1}^{n_g} W_{gq}^{-} {s_i^q}^*}- \beta_i^g {u_i^g}^*\right)  \right] \\
+& \sum_{i=1}^{n_c} \sum_{g= 1}^{n_g} \left(s_i^g - {s_i^g}^*\right)
\left[\beta_i^g u_i^g -\gamma_i^g \left(s^g_i - {c} \sum_{j=1}^{n_c} A_{ij}\left( s^g_j - s^g_i \right) \right)
-\left(\beta_i^g {u_i^g}^* 
-\gamma_i^g {s^g_i}^* 
+ {c} \sum_{j=1}^{n_c} A_{ij}\left( {s^g_j}^* - {s^g_i}^* \right) \right)\right]\\
=& \sum_{i=1}^{n_c} \sum_{g= 1}^{n_g} \left( u_i^g - {u_i^g}^* \right) 
\left[\alpha_i^g \left(R_i^g (s_i) - R_i^g (s_i^*) \right) - \beta_i^g \left( u_i^g -  {u_i^g}^*\right)  \right] \\
+& \sum_{i=1}^{n_c} \sum_{g= 1}^{n_g} \left(s_i^g - {s_i^g}^*\right)
\left[\beta_i^g \left( u_i^g -{u_i^g}^*\right) -\gamma_i^g \left(s^g_i - {s^g_i}^*\right)
+ {c} \sum_{j=1}^{n_c} A_{ij}\left( \left( s^g_j - {s^g_j}^*\right) -\left(s^g_i - {s^g_i}^*\right) \right) \right]\\
=& \sum_{i=1}^{n_c} \sum_{g= 1}^{n_g} 
\left[\alpha_i^g \left(R_i^g (s_i) - R_i^g (s_i^*) \right) \left( u_i^g - {u_i^g}^* \right)  - \beta_i^g \left( u_i^g -  {u_i^g}^*\right)^2  \right] \\
+& \sum_{i=1}^{n_c} \sum_{g= 1}^{n_g} 
\left[\beta_i^g \left( u_i^g -{u_i^g}^*\right) \left(s_i^g - {s_i^g}^*\right) -\gamma_i^g \left(s^g_i - {s^g_i}^*\right)^2\right]\\
+& \sum_{i=1}^{n_c} \sum_{g= 1}^{n_g} 
\left[ {c} \sum_{j=1}^{n_c} A_{ij}\left( \left( s^g_j - {s^g_j}^*\right) -\left(s^g_i - {s^g_i}^*\right) \right)\left(s_i^g - {s_i^g}^*\right) \right].
\end{aligned}    
\end{align}

We first examine the last term that arises due to the ``diffusive-like consensus'' among cells. Let $\Delta u_i^g = u^g_i - {u^g_i}^*$ and $\Delta s_i^g = s^g_i - {s^g_i}^*$. Since $A_{ij} = A_{ji}$,  we have
\begin{align}
\begin{aligned}
&{c} \sum_{i=1}^{n_c} \sum_{g= 1}^{n_g} 
\left[ \sum_{j=1}^{n_c} A_{ij}\left( \Delta s_j^g - \Delta s_i^g \right) \Delta s_i^g \right]   \\ 
&= {c}  \sum_{g= 1}^{n_g} 
\sum_{i=1}^{n_c} \sum_{j=1}^{n_c}\left[ \frac{1}{2} A_{ij}\left( \Delta s_j^g - \Delta s_i^g \right) \Delta s_i^g + \frac{1}{2} A_{ji}\left( \Delta s_i^g - \Delta s_j^g \right) \Delta s_i^g \right] \\
&= {c}  \sum_{g= 1}^{n_g} 
\sum_{i=1}^{n_c} \sum_{j=1}^{n_c}
\left[ -\frac{1}{2} A_{ij} \left( \Delta s_j^g - \Delta s_i^g \right)^2 \right] \\
&\leq  0.
\end{aligned}
\end{align}
For cell $i$, let $\Delta u_i = [\Delta u_i^1, \cdots, \Delta u_i^{n_g}]^\top$, and $\Delta s_i = [\Delta s_i^1, \cdots, \Delta s_i^{n_g}]^\top$. Then, we have
\begin{align}
\begin{aligned}
\dot{V}
\leq & 
\sum_{i=1}^{n_c} \sum_{g=1}^{n_g} 
\left[\alpha_i^g \left(R_i (s_i) - R_i (s_i^*) \right) \Delta u_i^g   - \beta_i^g \left(\Delta u_i^g\right)^2  \right] 
+ \sum_{i=1}^{n_c} \sum_{g= 1}^{n_g} 
\left[\beta_i^g \Delta u_i^g  \Delta s_i^g 
-\gamma_i^g \left(\Delta s_i^g\right)^2\right]\\
=& \sum_{i=1}^{n_c} \sum_{g=1}^{n_g} 
\alpha_i^g \left(R_i (s_i) - R_i (s_i^*) \right) \Delta u_i^g
+ \sum_{i=1}^{n_c}
\begin{bmatrix} \Delta u_i \\ \Delta s_i \end{bmatrix}^\top 
\begin{bmatrix} -\beta_i & \frac{1}{2} \beta_i \\ \frac{1}{2} \beta_i & -\gamma_i \end{bmatrix} 
\begin{bmatrix} \Delta u_i \\ \Delta s_i \end{bmatrix}.
\end{aligned}    
\label{eq.Lyapunov.multi.2}
\end{align}
We next focus on the first term $\sum_{i=1}^{n_c} \sum_{g=1}^{n_g} 
\alpha_i^g \left(R_i^g (s_i) - R_i^g (s_i^*) \right) \Delta u_i^g  \in \Rset$. Since $R_i^g(\cdot)$ is uniformly Lipschitz on the nonnegative orthant with coefficients $\omega_i$, we have
\begin{align}
\begin{aligned}
\sum_{i=1}^{n_c} \sum_{g=1}^{n_g} 
\alpha_i^g \left(R_i^g (s_i) - R_i^g (s_i^*) \right) \Delta u_i^g 
\leq& \sum_{i=1}^{n_c} \sum_{g=1}^{n_g} \left|\alpha_i^g \left(R_i^g (s_i) - R_i^g (s_i^*) \right) \Delta u_i^g \right| \\
\leq & \sum_{i=1}^{n_c} \sum_{g=1}^{n_g} \alpha_i^g \left| R_i^g (s_i) - R_i^g (s_i^*) \right| \left| \Delta u_i^g \right|\\
\stackrel{(a)}{\leq} & \sum_{i=1}^{n_c} \sum_{g=1}^{n_g} \alpha_i^g  \omega_i \vnorm{\Delta s_i}_2 \left| \Delta u_i^g \right|\\
\stackrel{(b)}{\leq} & \sum_{i=1}^{n_c} \sum_{g=1}^{n_g} \left(\frac{1}{4 \epsilon_i} {\alpha_i^g}^2 {\omega_i^g}^2 \vnorm{\Delta s_i}^2 +  \epsilon_i \left| \Delta u_i^g \right|^2 \right)\\
=& \sum_{i=1}^{n_c} \frac{1}{4 \epsilon_i}{\omega_i}^2
\left(\sum_{g=1}^{n_g} {\alpha_i^g}^2 \right) \vnorm{\Delta s_i}^2
+ \sum_{i=1}^{n_c} \epsilon_i \sum_{g=1}^{n_g} \left| \Delta u_i^g \right|^2 \\
\stackrel{(c)}{=}&  \sum_{i=1}^{n_c} \frac{1}{4 \epsilon_i}{\omega_i}^2 \vnorm{\alpha_i}_F^2
\vnorm{\Delta s_i}^2
+ \sum_{i=1}^{n_c} \epsilon_i \vnorm{\Delta u_i}_2^2 \\
=& \sum_{i=1}^{n_c} \begin{bmatrix} \Delta u_i \\ \Delta s_i \end{bmatrix}^\top 
\begin{bmatrix} \epsilon_i I_{n_g} & 0 \\ 0 & \frac{\omega_i^2}{4 \epsilon_i} \vnorm{\alpha_i}_F^2 I_{n_g} \end{bmatrix} 
\begin{bmatrix} \Delta u_i \\ \Delta s_i \end{bmatrix},
\end{aligned}
\end{align}
where (a) follows from the Cauchy-Schwarz inequality, and (b) follows from Young's inequality with arbitrary $\epsilon>0$, and (c) holds since $\sum_{g=1}^{n_g} {\alpha_i^g}^2 = \vnorm{\alpha_i}_F^2$.
Plugging the last expression into~\eqref{eq.Lyapunov.multi.2}, we get
\begin{align}
\begin{aligned}
\dot{V}(u,s)
\leq&  
\sum_{i=1}^{n_c} \begin{bmatrix} \Delta u_i \\ \Delta s_i \end{bmatrix}^\top 
\begin{bmatrix} -\beta_i + \epsilon_i I_{n_g} & \frac{1}{2} \beta_i \\
\frac{1}{2} \beta_i & -\gamma_i +\frac{\omega_i^2}{4 \epsilon_i} \vnorm{\alpha_i}_F^2 I_{n_g} 
\end{bmatrix} 
\begin{bmatrix} \Delta u_i \\ \Delta s_i \end{bmatrix}.
\end{aligned}    
\end{align}
For each cell $i$, let $Q_i \in \Rset^{2n_g \times 2 n_g}$ be defined as
\begin{align}
Q_i:= \begin{bmatrix} \beta_i - \epsilon_i I_{n_g} & -\frac{1}{2} \beta_i \\
- \frac{1}{2} \beta_i & \gamma_i - \frac{\omega_i^2}{4 \epsilon_i} \vnorm{\alpha_i}_F^2 I_{n_g}
\end{bmatrix}. 
\end{align}
If $Q_i$ is positive definite for all cells $i$, then $\dot{V} \leq - \sum_i^{n_c} \begin{bmatrix} \Delta u_i \\ \Delta s_i \end{bmatrix}^\top  Q_i \begin{bmatrix} \Delta u_i \\ \Delta s_i \end{bmatrix} < 0$ whenever $\begin{bmatrix}\Delta u_i \\ \Delta s_i \end{bmatrix} \neq 0$. Since $Q_i$ is symmetric, it is positive definite if and only if $\beta_i - \epsilon_i I_{n_g} \succ 0$, and its Schur complement in $Q_i$ satisfies
\begin{align}
\gamma_i - \frac{\omega_i^2}{4 \epsilon} \vnorm{\alpha_i}_F^2 I_{n_g}
- \frac{1}{4} \beta_i \left(\beta_i - \epsilon I_{n_g} \right)^{-1} \beta_i \succ 0.
\end{align}

Choose $\epsilon_i = \frac{\omega_i \vnorm{\alpha_i}_F }{2}$. Since $\gamma_i,\beta_i$ are diagonal matrices, the condition becomes
\begin{align}
\beta_i^g > \frac{\omega_i \vnorm{\alpha_i}_F }{2}, \quad
\gamma_i^g >  \frac{\omega_i \vnorm{\alpha_i}_F }{2}
+ \frac{({\beta_i^g})^2}{4(\beta_i^g - \frac{\omega_i \vnorm{\alpha_i}_F }{2})}.
\end{align}
This completes the proof.

\subsection{Proof of Theorem~\ref{thm.consensus}}
From the definition of $\tilde{s}^g$, we have
\begin{align}
\begin{aligned}
\frac{d \tilde{s}^g}{dt}   
&= \frac{d s^g}{dt} - \frac{d \overline{s}^g}{dt} \\
&= B^g u^g(t) -\Gamma^g s^g(t) - {c} L s^g(t)
- P_c \left( B^g u^g(t) -\Gamma^g s^g(t) \right) \\
&= \left(I_{n_c} - P_c\right) \left( B^g u^g(t) -\Gamma^g s^g(t) \right)
- {c} L s^g(t).
\end{aligned}
\end{align}
Note that since $L \bm{1} = 0$, we have
\begin{align}
\begin{aligned}
L \tilde{s}^g
= L \left( s^g - \overline{s}^g \right) 
= L \left( s^g -    \frac{1}{n_c} \bm{1} \bm{1}^\top s^g \right) 
= L s^g.
\end{aligned}
\end{align}
Hence, the dynamics of the deviation is of the form
\begin{align}
\frac{d \tilde{s}^g}{dt}=z^g(t) + {c} L \tilde{s}^g(t),
\end{align}
where $z^g(t):= \left(I_{n_c} - P_c\right) \left( B^g u^g(t) -\Gamma^g s^g(t) \right)$. By definition of $P_c$, we have $P_c^2= P_c$ and hence $(I-P_c)^2 = I-P_c$. Thus, 
\begin{align}
\begin{aligned}
\vnorm{z^g(t)}_2^2
=& \vnorm{\left(I_{n_c} - P_c\right) \left( B^g u^g(t) -\Gamma^g s^g(t) \right)}_2^2 \\
\leq& \vnorm{B^g u^g(t) -\Gamma^g s^g(t)}_2^2.
\end{aligned}
\end{align}
Since $u^g(t)$,$s^g(t)$ are uniformly bounded for all $t \ge 0$ for each gene $g$, the norm of $z^g(t)$ is bounded by $Z^g_m$.

Now, consider the time derivative of $\Psi(t) = \frac{1}{2}\left\| \tilde{s}^g(t)\right\|_2^2$:
\begin{align}
\begin{aligned}
\frac{d \Psi}{dt}  
&= (\tilde{s}^g)^\top \left( z^g(t) - {c} L \tilde{s}^g(t)\right)\\
&\stackrel{(a)}{\leq} \vnorm{\tilde{s}^g}_2 \vnorm{z^g(t)}_2
-{c}(\tilde{s}^g)^\top L \tilde{s}^g(t)\\
&\stackrel{(b)}{\leq} \vnorm{\tilde{s}^g}_2 \vnorm{z^g(t)}_2
-{c \lambda_2(L)} \vnorm{\tilde{s}^g}_2^2\\
&=\left(2\Psi(t)\right)^{1/2} \vnorm{z^g(t)}_2
-{2 c \lambda_2(L)} \Psi(t) \\
&\stackrel{(c)}{\leq} 2 \epsilon \Psi(t) + \frac{1}{4\epsilon}\vnorm{z^g(t)}_2
-{2 c \lambda_2(L)} \Psi(t).
\end{aligned}
\end{align}
where (a) follows from the Cauchy-Schwarz inequality, (b) holds since $\tilde{s}^g$ is orthogonal to the one-dimensional subspace spanned by $\bm{1}$ and hence $\lambda_2(L) \vnorm{\tilde{s}^g(t)}_2^2 \leq (\tilde{s}^g)^\top L \tilde{s}^g(t)$, and (c) follows from Young's inequality. 
Setting $\epsilon = \frac{c \lambda_2(L)}{2} $, we have
\begin{align}
\begin{aligned}
\frac{d \Psi}{dt}  
&\leq {c \lambda_2(L) } \Psi(t) + \frac{1}{2c \lambda_2(L)}Z_m^g
- {2c \lambda_2(L)} \Psi(t) \\
&= - {c \lambda_2(L)}\Psi(t) + \frac{1}{ c \lambda_2(L)} Z_m^g.
\end{aligned}
\end{align}
Hence,
\begin{align}
\begin{aligned}
\Psi(t) 
& \leq  \Psi(0) e^{- {c \lambda_2(L)} t}
+ \frac{1}{c \lambda_2(L)}Z_m^g 
\int_0^t e^{- {c \lambda_2(L)} (t-\tau)}  d \tau \\
&\leq  \Psi(0) e^{- {c \lambda_2(L)} t}
+ \frac{1}{2c \lambda_2(L)}Z_m^g 
\left(1 - e^{- {c \lambda_2(L)}t} \right) \\
&\leq  \Psi(0) e^{- {c \lambda_2(L)}t}
+ \frac{1}{2c \lambda_2(L)}Z_m^g. 
\end{aligned}
\end{align}

Letting $t \to \infty$, and plugging $\Psi(t) = \frac{1}{2}\left\| \tilde{s}^g(t)\right\|_2^2$ back into the formula, we get
\begin{align}
\limsup_{t \to \infty} \vnorm{\tilde{s}^g}_2^2 \leq \frac{1}{c \lambda_2(L)}Z_m^g.
\end{align}
This completes the proof.

\end{document}